\documentclass[11pt]{amsart}
\usepackage{enumerate}
\usepackage{amsmath}
\usepackage{amssymb}
\usepackage{mathrsfs}
\usepackage{tikz}
\usepackage{tikz-cd}
\usepackage{comment}
\usepackage{color}
\usepackage[colorlinks,citecolor=blue,urlcolor=black,linkcolor=black]{hyperref}

\newtheorem{theorem}{Theorem}[section]
\newtheorem{maintheorem}{Main Theorem}
\newtheorem*{theorem*}{Theorem}
\newtheorem{claim}[theorem]{Claim}

\newtheorem{sclaim}[theorem]{Subclaim}
\newtheorem{fact}[theorem]{Fact}
\newtheorem{lemma}[theorem]{Lemma}

\newtheorem{proposition}[theorem]{Proposition}
\newtheorem{corollary}[theorem]{Corollary}

\theoremstyle{convention}

\theoremstyle{definition}
\newtheorem{setup}{Setup}
\newtheorem{definition}[theorem]{Definition}
\newtheorem{conv}[theorem]{Convention}

\newtheorem{question}[theorem]{Question}
\newtheorem*{question*}{Question}
\theoremstyle{remark}
\newtheorem{remark}[theorem]{Remark}

\newtheorem{notation}[theorem]{Notation}

\newcommand{\lusim}[1]{\smash{\underset{\raisebox{1.2pt}[0cm][0cm]{$\sim$}}
{{#1}}}}
\def\l{{\langle}}
\def\r{{\rangle}}

\def\forces{\Vdash}
\def\l{\langle}
\def\r{\rangle}
\def\Succ{{\rm Succ}}

\newcommand{\s}{\subseteq}
\newcommand\diagonal{\bigtriangleup}

\newcommand{\dom}{\mathrm{dom}}

\newcommand{\cf}{{\rm cf}}

\newcommand{\crit}{\mathrm{crit}}
\newcommand\cat{{}^\curvearrowright}
 \newcommand{\one}{\mathop{1\hskip-3pt {\rm l}}}

\newcount\skewfactor
\def\mathunderaccent#1#2 {\let\theaccent#1\skewfactor#2
\mathpalette\putaccentunder}
\def\putaccentunder#1#2{\oalign{$#1#2$\crcr\hidewidth
\vbox to.2ex{\hbox{$#1\skew\skewfactor\theaccent{}$}\vss}\hidewidth}}

\def\smallbox#1{\leavevmode\thinspace\hbox{\vrule\vtop{\vbox
   {\hrule\kern1pt\hbox{\vphantom{\tt/}\thinspace{\tt#1}\thinspace}}
   \kern1pt\hrule}\vrule}\thinspace}

\newcommand{\Add}{\mathrm{Add}}

\newcommand\ale[1]{\marginpar{Alejandro: #1}}
\newcommand\tom[1]{\marginpar{Tom: #1}}

\let\labeloriginal\label
\let\reforiginal\ref
\def\ref#1{\reforiginal{#1}}
\def\label#1{\labeloriginal{#1}}

\title[Non-normal]{Non-Normal Magidor-Radin types of forcings}
\author[Benhamou]{Tom Benhamou}
\address[Benhamou]{Department of Mathematics, Rutgers University, Piscataway (NJ) 08854-
8019, USA.}
\email{tom.benhamou@rutgers.edu}
\thanks{The research of the first author was supported by the National Science Foundation under Grant
No. DMS-2346680. The research of the second author was supported by the Department of Mathematics and the Center of Mathematical Sciences and Applications at Harvard University.}
\author[Poveda]{Alejandro Poveda}
\address[Poveda]{Harvard University, Department of Mathematics and Center of Mathematical Sciences and Applications, Cambridge (MA), 02138, USA}
\email{alejandro@cmsa.fas.harvard.edu}

\begin{document}

\begin{abstract}
    We develop the non-normal variations of two classical Pri\-kry-type forcings; namely,  Magidor and Radin forcings. We generalize the fact that the non-normal Prikry forcing is a projection of the extender-based to a coordinate of the extender to our forcing and the Radin/Magidor-Radin-extender-based forcing from \cite{CarmiMagidorRadin,CarmiRadin}. \\
 Then, we 
    show that both the non-normal variation of Magidor and Radin forcings can add a Cohen generic function to every limit point of cofinality $\omega$ of the generic club. Second,  we show that these phenomenon is limited to the cases where the forcings are not designed to change the  cofinality of a measurable $\kappa$ to $\omega_1$. Specifically, in the above-mentioned circumstances these forcings do not project onto any $\kappa$-distributive forcing. We use that to conclude that the extender-based Radin/Magidor-Radin forcing does not add fresh subsets to $\kappa$ as well. In the second part of the paper we focus on the natural non-normal variation of Gitik's forcing from \cite[\S3]{GitikNonStationary}. Our main result shows that this poset can be employed to change the cofinality of a measurable cardinal $\kappa$  to $\omega_1$ while introducing a Cohen subset of $\kappa$.
\end{abstract}
\maketitle
\section{Introduction}
Singular Cardinal Combinatorics is a prominent area of research in mo\-dern set theory. The field is primarily concerned with the properties of singular cardinals and its small successors (such as $\aleph_\omega$ and $\aleph_{\omega+1}$) and how these change across the set-theoretic multiverse. During the last fifty years, research in this field have yielded some of the most sophisticated technologies ever invented in set theory. A paradigmatic example are the so-called \emph{Prikry-type forcings}. The field was pioneered by Prikry \cite{Prikry} who provided the first example of a forcing poset changing the cofinality of a measurable cardinal to $\aleph_0$ without collapsing cardinals. However, it was  Magidor who  through a series of groundbreaking discoveries \cite{MagIdentity, MagSing, MagSingII} placed Prikry-type posets in the spotlight. Other major results employing these forcings were obtained by Cummings and Woodin \cite{CumGCH},  Gitik \cite{Gitikstrength, GitikNonStationary} and Foreman and Woodin \cite{ForWoo}.

\smallskip

By nowadays Prikry-type forcings count with a beautiful and extensive theory, mostly developed and accounted by Gitik in \cite{Gitik-handbook}. During the last few decades the abstract study of Prikry-type forcings has became a topic of central interest in set theory. Like any other important mathematical structure $\mathcal{M}$, the problem of classifying  the substructures of $\mathcal{M}$ is of high importance. In that direction, a problems which has elicited a major interest concerns the possible intermediate models of a generic extension by a Prikry-type forcing. This amounts to asking whether a given Prikry-type forcing $\mathbb{P}$ projects onto other Prikry-type posets or even onto other classical forcings -- the epitome of these latter being Cohen forcing.

The following is a succinct  account of what is known for Prikry forcing and its tree-like variations where these posets are defined using normal ultrafilters. First,  Gitik, Koepke, and Kanovei \cite{PrikryCaseGitikKanKoe} proved that any intermediate model of a generic extension by Prikry forcing with a normal ultrafilter must be a generic extension by Prikry forcing (with the same normal ultrafilter). In contrast, Koepke, Rasch and Schlicht \cite{MinimalPrikry}  construted a Tree Prikry forcing   yielding a minimal forcing extension -- this  phenomenon is akin to the classical \emph{Sacks property} of Sacks forcing.  
More recently, Benhamou and Gitik \cite{TomMaster}, and afterwards Benhamou, Gitik, and Hayut \cite{TomYairMoti}, proved  that the classical Tree Prikry forcing with non-normal ultrafilters can project onto a wide variety of $\kappa$-distributive forcings of cardinality $\kappa$ -- including $\Add(\kappa,1)$. In addition, that paper provides a non-trivial large-cardinal lower bound for this forcing to project onto every $\kappa$-distributive (even ${<}\kappa$-strategically closed) poset of cardinality $\kappa$. The results in \cite{TomYairMoti} are a sequel of a classical theorem of  Gitik saying that  Supercompact Prikry forcing can be arranged to project onto every $\kappa$-distributive forcing \cite[\S6.4]{Gitik-handbook} of cardinality $\kappa$. Finally, Benhamou and Gitik \cite{onCohenandPrikry} constructed an ultrafilter $U$ such that Prikry forcing with $U$ projects onto $\Add(\kappa,\kappa^+)$, showing that the class of distributive forcings onto which the Tree Prikry forcing projects exceeds those of cofinality $\kappa$.

In the context of Magidor/Radin-like forcings -- again, relative to normal ultrafilters -- our knowledge is way more narrow. Fuchs \cite{fuchs} proved that if $c,d$ are generic sequences for  Magidor forcing of \cite{ChangeCofinality} and $c\in V[d]$ then $c$ is almost contained in $d$. Benhamou and Gitik \cite{TomMaster,partOne,TomMotiII} generalized Gitik-Koepke-Kanovei's result and provided a full characterization of the intermediate models of a generic extension by the Mitchell version from \cite{MITCHELLHowWeak} of Magidor forcing  relative to a coherent sequence of measures  with $o(\kappa)<\kappa^+$. Namely, if  $G$ is generic for the Magidor/Radin forcing then  every intermediate model $V\subseteq M\subseteq V[G]$ is of the form $V[C]$ where $C$ is a subset of the generic club added by $G$. In the case where $o(\kappa)<\kappa$, models of the form $V[C]$ are generic for a finite iteration of Magidor-like forcings. 

\smallskip

The above results indicate that in the \emph{normal context} one should not expect a rich variety of intermediate extensions for a given Prikry-type forcing, while in the \emph{non-normal context} special constructions can provide a richer variety. This reflection invites to developing variations of the aforementioned forcings when the ultrafilters involved are non-normal.

\smallskip

In this paper we develop two new Prikry-type technologies -- the Mitchell-style non-normal Magidor forcing and the non-normal Radin forcing, respectively. Versions of these forcings appeared somewhat implicitly in Merimovich's works on Extender-Based Magidor/Radin forcing \cite{CarmiRadin,CarmiMagidorRadin}.

Let $\mathbb{M}[\vec{U}]$ (resp. $\mathbb{R}_u$) denote our non-normal version of the Mitchel-style Magidor (resp. Radin forcing) with respect to a generalized coherent sequence of ultrafilters\footnote{See Definition~\ref{GeneralizedCoherent}.} $\vec{U}$ (resp. a measure sequence $u$) of length $\omega_1$. These two posets will be respectively developed in \S\ref{Sec: non-normalMagidor} and \S\ref{SectionNonNormalRadin} of this paper. Later in  \S\ref{sec: AddingCohends} we shall employ them to demonstrate that they yield  garden-variety of intermediate generic extensions. The  mathematical meaning of this assertion is make precise by our first main theorem:

\begin{maintheorem}\label{thm: main 1}
It it consistent for both $\mathbb{M}[\vec{U}]$ and $\mathbb{R}_u$ to yield a club $C\s \kappa$ of cardinals with $\mathrm{otp}(C)=\omega_1$ such that every limit point of $\alpha\in C$ carries a Cohen generic function for $\Add(\alpha,1)$.
\end{maintheorem}
Therefore, in the above model, both $\mathbb{M}[\vec{U}]$ and $\mathbb{R}_u$ project onto $\Add(\alpha,1)$ for every limit point $\alpha\in C.$ 
This fact is optimal in the sense that one cannot hope for these forcings to 
project onto Cohen forcing $\Add(\alpha,1)$ for a singular cardinal $\alpha$ of uncountable cofinality in the eventual Prikry-type extension.  This conclusion will be inferred as a consequence of these posets  not adding fresh subsets to $\kappa$ (Corollary~\ref{cor: no fresh subsets}). Moreover, we show that the same conclusion is applicable Merimovich's Extender-Based Radin and  Magidor/Radin forcings from \cite{CarmiRadin,CarmiMagidorRadin} (see Corollary~\ref{cor: no fresh subsets carmi}).

\smallskip

So, is it possible for a Prikry-type forcing $\mathbb{P}$ to project onto $\Add(\kappa,1)$ when $\kappa$ is a measurable cardinal that changes its cofinality to $\omega_1$ after forcing with $\mathbb{P}$? We show that the answer  is affirmative but these requires fairly different methods to be established. Specifically, we show that the  non-normal variation of Gitik's forcing $\mathbb{P}(\kappa,\omega_1)$ from \cite{GitikNonStationary} does the job. Unlike the previously mentioned posets, $\mathbb{P}(\kappa,\omega_1)$    changes the cofinality of a measurable cardinal $\kappa$ with $o(\kappa)=\omega_1$ without introducing bounded subsets to $\kappa$. This poset is defined over a generic extension of $V$ by an Easton-supported (a.k.a., \emph{Gitik iteration}) of Prikry-type forcings (see \cite{GitikNonStationary}). 

Our main result in regards to $\mathbb{P}(\kappa,\omega_1)$ reads as follows:

\begin{maintheorem}\label{thm: main2}
    It is consistent for $\mathbb{P}(\kappa,\omega_1)$ to project onto $\Add(\kappa,1).$
\end{maintheorem}

In a recent paper \cite{KaplanGitik}, Gitik and Kaplan have proved that certain iteration of Prikry-type forcings of length $\kappa$ do not add fresh subsets to $\kappa$. In particular these results apply to the preparatory iteration in Theorem \ref{thm: main2}, which therefore do not add fresh sets to $\kappa$. 
\smallskip

The structure of the paper is as follows. We begin with \S\ref{sec: non-normal coherent} discussing two non-normal variations of the classical notion of coherent sequence of normal ultrafilters. This analysis is used later in \S\ref{Sec: non-normalMagidor} where we present  the Mitchel-styled Magidor forcing $\mathbb{M}[\vec{U}]$ with respect to a generalized coherent sequence $\vec{U}$. In this section we also show that this forcing can be recasted as a projection of Merimovich Extender-Based posets  \cite{CarmiRadin,CarmiMagidorRadin}. In \S\ref{SectionNonNormalRadin} we present the non-normal Radin forcing and in \S\ref{sec: AddingCohends} we prove \textbf{Main Theorem}~\ref{thm: main 1}. In \S\ref{sec: Gitik} we discuss the non-normal version of Gitik's forcing following \cite{GitikNonStationary} and prove \textbf{Main Theorem}~\ref{thm: main2}. The manuscript is concluded with \S\ref{sec: open questions} by drawing possible future directions and proposing a few open questions.

\begin{conv}
    Given $U$ a $\kappa$-complete ultrafilter over $\kappa$ we will tend to denote either by $M_U$ or $\mathrm{Ult}(V,U)$ the transitive collapse of the ultrapower of $V$ by $U$. Similarly, the induced elementary embedding from $V$ to $M_U$ will be denoted by $j_U.$ 
    When it comes to a forcing posets  we shall stick to the \emph{Israeli convention}; namely, when we write $p\leq q$ we will be meaning that $q$ is \emph{stronger} (i.e., more informative) than $p$. Given a regular cardinal $\kappa$ we shall denote by $\Add(\kappa,1)$ the \emph{Cohen forcing} at $\kappa$; namely, conditions in $\Add(\kappa,1)$ are partial functions $p\colon \kappa\rightarrow 2$ with $|p|<\kappa$ ordered by $\s$-extension. Whenever $U$ is a non-normal $\kappa$-complete over $\kappa$ we will denote by $\mathbb{T}_U$ the \emph{Tree-Prikry forcing} relative to $U$ (see \cite[\S1]{Gitik-handbook}).
\end{conv}

\section{Non-normal coherent sequences}\label{sec: non-normal coherent}
Let us fix $\kappa$ a measurable cardinal. Given two $\kappa$-complete (non-trivial) ultrafilters $U, W$ over $\kappa$ we shall say that \emph{$U$ is Mitchell below $W$} and write $U\triangleleft W$ whenever $U\in \mathrm{Ult}(V,W).$ Certainly, this is the natural generalization of the classical \emph{Mitchell order} $\triangleleft$ between normal measures \cite{MitchHandbook}.


We define two types of coherent sequences of ultrafilters; namely, \emph{generalized coherent sequences} (Definition~\ref{GeneralizedCoherent}) and \textit{almost coherent sequences} (Definition~\ref{Def:almostcoherentesequence}). 
\begin{definition}\label{GeneralizedCoherent}
    A sequence $$\vec{U}=\l U(\alpha,i)\mid \alpha<\kappa,i<o^{\vec{U}}(\alpha)\r^{\smallfrown}\l U(\kappa,i)\mid i<\gamma\r$$ is a \emph{generalized coherent sequence of length $\gamma$ with a top cardinal $\kappa$} if:
    \begin{enumerate}[(i)]
        \item  There is a function $\pi:\kappa\rightarrow\kappa$ such that for every $i<\gamma$ $[\pi]_{U(\kappa,i)}=\kappa$.
        \item For each $\alpha<\kappa$ and $\beta<o^{\vec{U}}(\alpha)$, $U(\alpha,\beta)$ is a $\pi(\alpha)$-complete ultrafilter over $\pi(\alpha)$. Also for  $i<\gamma$, $U(\kappa,i)$ is a $\kappa$-complete ultrafilter over $\kappa$.
        \item For each $\alpha<\kappa$ and  $i<o^{\vec{U}}(\alpha)$, $[\pi\restriction \pi(\alpha)]_{U(\alpha,i)}=\pi(\alpha)$.
        \item For every $\alpha\leq\kappa$ and $\beta<o^{\vec{U}}(\alpha)$, $$j_{U(\alpha,\beta)}(\vec{U})([\mathrm{id}]_{U(\alpha,\beta)})=\l U(\alpha,i)\mid i<\beta\r,$$ where $\vec{W}(\gamma)$ denotes the values of the sequence $\vec{W}$ at $\gamma$; i.e., $$\l W(\gamma,i)\mid i<o^{\vec{W}}(\gamma)\r.$$
        \end{enumerate}
        We say that $\vec{U}$ is \textit{special}, if \begin{enumerate}
            \item [(v)] whenever $U(\alpha,i)$ is non-normal,  $j_{U(\alpha,i)}(\vec{U})(\pi(\alpha))=\l\r$.
    
        \end{enumerate}
\end{definition}
\begin{remark}
   If all the measures in $\vec{U}$ are normal then one recovers the standard notion of a coherent sequence of measures \cite{MitchHandbook}.
\end{remark}

\begin{theorem}
    Assume the $\mathrm{GCH}$ holds. Suppose that $\l U_i\mid i<\gamma\r$ (with $\gamma<\kappa$) is a $\triangleleft$-increasing sequence of $\kappa$-complete ultrafilters over $\kappa$ (which are not necessarily normal). Then there is a generalized coherent sequence $\vec{U}$ of length $\gamma$ with a top cardinal $\kappa$ such that $U(\kappa,i)=\gamma$ for every $i<\gamma$. Moreover if
    $$\{U_i\mid i<\gamma\}\cup\{U_i^{\text{nor}}\mid U_i\text{ is not normal}\}$$
    are distinct ultrafilters, then we can ensure that the sequence is special.
\end{theorem}
\begin{proof}
Start by finding a sequence of sets $$\mathcal{A}=\{ A_\alpha\mid \alpha<\gamma\}\cup\{ A_\alpha'\mid U_\alpha\text{ is non-normal}\}$$ such that:
\begin{enumerate}
    \item For all $\alpha<\gamma$, $A_\alpha\in U_\alpha$ and $\min(A_\alpha)>\gamma$.
    \item If $A,B\in\mathcal{A}$ and $A\neq B$ then $A\cap B=\emptyset$.
    \item If $U_\alpha$ is non-normal then  $A_\alpha'\in U_{\alpha}^{\text{nor}}$.
\end{enumerate}
Such a sequence exists if $$\{U_i\mid i<\gamma\}\cup\{U_i^{\text{nor}}\mid U_i\text{ is not normal}\}$$
is a set of less than $\kappa$-many distinct $\kappa$-complete ultrafilter.
Otherwise, we just require that 
$\mathcal{A}=\{ A_\alpha\mid \alpha<\gamma\}$ and ignore $(3)$.
Next, find $\pi:\kappa\rightarrow\kappa$ such that for every $i<\gamma$, $[\pi]_{U_i}=\kappa$, and define by induction on $\alpha<\gamma$, $\vec{V}^{(\alpha)}$, such that:
\begin{enumerate}
    \item $\dom(V^{(0)})=\{\l\kappa,0\r\}$ and $V^{(0)}(\kappa,0)=U_0$.
    \item $\vec{V}^{(\alpha)}$ is a generalized coherent sequence of length $\alpha+1$ with a top cardinal $\kappa$.
    \item $\alpha<\beta<\omega_1\Rightarrow\vec{V}^{(\alpha)}\subseteq \vec{V}^{(\vec{\beta})}$ (as partial functions).
    \item For $\alpha>0$, $\dom(\vec{V}^{(\alpha)})\setminus \bigcup_{\beta<\alpha}\dom(V^{(\vec{\beta})})=B_{\alpha}\times\alpha\cup \{\l\kappa ,\alpha\r\}$, where $B_{\alpha}\subseteq A_{\alpha}$ and $B_{\alpha}\in U_{\alpha}$.
    \item $V^{(\alpha)}(\kappa,\alpha)=U_\alpha$.
    \item For every $(\eta,i)\in \dom(\vec{V}^{(\alpha)})$, $ B_i\cap\pi(\eta)\in V^{(\alpha)}(\eta,i)$
\end{enumerate}
In the moreover case we also require that:
\begin{enumerate}
    
    \item [(7)]  $A'_i\cap \pi(\eta)\in \pi^\eta_*(V^{(\alpha)}(\eta,i))$ for all $i$ such that $V^{(\alpha)}(\eta,i)$ is non-normal and where $\pi^\eta=\pi\restriction\eta$.
\end{enumerate}
The following claim says that it suffices to construct the sequence above.
\begin{claim}
    Let $(\vec{V}^{(\alpha)})_{\alpha<\gamma}$ be a sequence satisfying $(1)-(6)$ as above and let $\beta\leq\gamma$, then $\vec{V}=\bigcup_{\alpha<\beta}\vec{V}^{(\alpha)}$ is a coherent sequence of length $\beta$ with a top cardinal $\kappa$. Moreover, if $(7)$ holds than the sequence is special.
\end{claim}
\begin{proof}[\textit{Proof of claim.}] By (1),(3),(4), $\dom(\vec{V})=(\bigcup_{0<\beta<\alpha}B_\beta\times\beta)\cup \{\kappa\}\times \alpha$ and $\vec{V}\restriction \dom(\vec{V}^{(\alpha)})=\vec{V}^{(\alpha)}$. Hence (i)-(iii) are trivial.

    To see (iv), let $(\eta,\beta)\in \dom(\vec{V})$ and there is $\alpha<\gamma$ such that $(\eta,\beta)\in \dom(V^{(\alpha)})$, then $\beta\leq \alpha$ and $V(\eta',i)=V^{(\alpha)}(\eta',i)$ for every $\eta'\in B_{\beta'}$ and $i<\beta'$, for some $\beta'\leq\beta$. Hence for every $\rho\in B_\beta\cap\pi(\eta)$, $\vec{V}(\rho)=\vec{V}^{(\alpha)}(\rho)$. By $(6)$, $B_\beta\cap \pi(\eta)\in V(\eta,\beta)$ and therefore $j_{V(\eta,\beta)}(\vec{V})([id]_{V(\eta,\beta)})=j_{V(\eta,\beta)}(\vec{V}^{(\alpha)})([id]_{V(\eta,\beta)})$. Using the coherency of $\vec{V}^{(\alpha)}$, $$j_{V(\eta,\beta)}(\vec{V})([id]_{V(\eta,\beta)})=j_{V^{(\alpha)}(\eta,\beta)}(\vec{V}^{(\alpha)})([id]_{V^{(\alpha)}(\eta,\beta)})=$$
    $$=\l V^{(\alpha)}(\eta,i)\mid i<\beta\r=\l V(\eta,i)\mid i<\beta\r.$$
    Finally, to see (v), we note that if $V(\eta,\beta)$ is non-normal, then by $(7)$ $A'_\beta\cap \pi(\eta)\in \pi^\eta_*(V^{(\alpha)}(\eta,\beta))$ and since this set is disjoint from the domain of $\vec{V}^{(\alpha')}$ for every $\alpha'<\gamma$, it is disjoint from $\dom(\vec{V})$ and therefore $\pi(\eta)\notin\dom( j_{V(\eta,\beta)}(\vec{V}))$. We conclude that $j_{V(\eta,\beta)}(\vec{V})(\pi(\eta))=\l\r$.
\end{proof} 

Let us turn to the inductive definition of $
\vec{V}^{(\alpha)}$, let $\dom(\vec{V}^{(0)})=\{\l\kappa,0\r\}$ and $\vec{V}^{(0)}(\kappa,0)=U_0$. Now suppose that $V^{(\beta)}$ has been defined for $\beta<\alpha$. By the previous claim, letting $\vec{V}=\bigcup_{\beta<\alpha}\vec{V}^{(\beta)}$, we have that $\vec{V}$ is a generalized coherent sequence of length $\alpha$ with a top cardinal $\kappa$ and $\dom(\vec{V})=(\bigcup_{0<\beta<\alpha}B_\beta\times\beta)\cup \{\kappa\}\times \alpha$. By $(2)$, we let $\vec{V}^{(\alpha)}\restriction\dom(\vec{V})=\vec{V}$, and by $(5)$, we have to define $V^{(\alpha)}(\kappa,\alpha)=U_\alpha$. By $(4)$, it remains to define $B_\alpha\subseteq A_\alpha$ and $V^{(\alpha)}\restriction B_\alpha\times\alpha$. Towards this, since $\alpha<\kappa$ and since we started with a Mitchell increasing sequence of ultrafilters, we have $\l U_i\mid i<\alpha\r\in M_{U_{\alpha}}$, hence we can find a function such that $\l U_i\mid i<\alpha\r=[\eta\mapsto \l V^\eta_i\mid i<\alpha\r]_{V_{\alpha}}$. Also, 
\begin{enumerate}
    \item [(a)] $M_{U_{\alpha}}\models\vec{V}=(j_{U_{\alpha}}(\vec{V})\restriction\kappa)^{\smallfrown}\l U_i\mid i<\alpha\r\text{ is coherent}$.
    \item [(b)] For $i<\alpha$,  $M_{U_{\alpha}}\models j_{U_{\alpha}}(A_i)\cap \kappa=A_i\in U_i.$
    \item [(c)]
    $M_{U_\alpha}\models$ if $U_i$ is non-normal then  $j_{U_\alpha}(A'_i)\cap\kappa= A'_i\in \pi_*(U_i)=(j_{U_\alpha}(\pi)\restriction\kappa)_*(U_i)$.
    \item [(d)] For $i<\alpha$,  $M_{U_{\alpha}}\models [j_{U_\alpha}(\pi)\restriction\kappa]_{U_i}=[\pi]_{U_i}=\kappa=j_{U_\alpha}(\pi)(\kappa)$.
\end{enumerate}
  Reflecting this, we can find a set $B_{\alpha}\in V_\alpha$ such that for every $\eta\in B_\alpha$,
  \begin{enumerate}
      \item [(a)] $\vec{V}\restriction\pi(\eta)^{\smallfrown}\l V^\eta_0,...,V^\eta_\alpha\r\text{ is coherent with a top cardinal }\pi(\eta)$.
      \item [(b)] For $i<\alpha$, $ A_i\cap\pi(\eta)\in V^{\eta}_i$.
      \item [(c)] $A_i\cap\pi(\eta)\in \pi^\eta_*(V^\eta_i)$ if $V^{\eta}_i$ is non-normal.
      \item [(d)] For $i<\alpha$, $[\pi\restriction\pi(\eta)]_{V^\eta_i}=\pi(\eta)$. 
  \end{enumerate} 

For $\eta\in B_{\alpha}$ and $i<\alpha$, let $V^{(\alpha)}(\eta,i)=V^\eta_{i}$.

Let us check $(1)-(7)$. First $(1),(3),(4),(5)$ are trivial. Condition $(6),(7)$ follows from the induction hypothesis and conditions $(b),(c)$ above. It remains to check $(2)$, i.e. that $V^{(\alpha)}$ is a generalized coherent sequence: (i)-(iii) follows directly from the construction. 

To see (iv), let $(\eta,\beta)\in\dom(\vec{V}^{(\alpha)})$. If $(\eta,\beta)\in \dom(\vec{V})$, then $V^{(\alpha)}(\eta,\beta)=V(\eta,\beta)$. Note that $B_{\alpha}\cap \pi(\eta)\notin V^{(\alpha)}(\eta,\beta)$ (which are the only cardinals where we made changes below $\pi(\eta)$ in  $\vec{V}^{(\alpha)}$) and thus 
$$j_{V(\eta,\beta)}(\vec{V}^{(\alpha)})([id]_{V(\eta,\beta)})=j_{V(\eta,\beta)}(\vec{V})([id]_{V(\eta,\beta)}).$$
 By the induction hypothesis, we have that $$j_{V(\eta,\beta)}(\vec{V})([id]_{V(\eta,\beta)})=\l V(\eta,i)\mid i<\beta\r=\l V^{(\alpha)}(\eta,i)\mid i<\beta\r,$$ and so we are done.
If $(\eta,\beta)\in \dom(\vec{V}^{(\alpha)})\setminus \dom(\vec{V})$, then either $\eta\ni B_\alpha$, in which case,  by $(a)$, $\vec{V}\restriction \eta^{\smallfrown}\l V^\eta_i\mid i<\alpha\r$ is coherent. Again, $\vec{V}^{(\alpha)}\restriction\eta$ defers from $\vec{V}^{(\alpha)}\restriction\eta$ only on $B_{\alpha}\cap \eta$ which is measure $0$ with respect to $V(\eta,\beta)$, for every $i<\alpha$. Hence the ultrapower by $V(\eta,\beta)$ will still satisfy the coherency requirement in (iv). The case $\eta=\kappa$ is similar. 

Finally to see (v), use $(c)$ and note that for every $(\eta,\beta)$ for which $V^{(\alpha)}(\eta,\beta)$ is non-normal, $A'_\beta\cap \pi(\eta)\in \pi^\eta_*(V^{(\alpha)}(\eta,\beta))$. This set is disjoint from the domain of $\vec{V}^{(\alpha)}$ and therefore \begin{equation*}
    j_{V^{(\alpha)}(\eta,\beta)}(\vec{V}^{(\alpha)})(\pi(\alpha))=\emptyset.\qedhere
\end{equation*}
\end{proof}
We will need also a particular case of a generalized coherent sequence which we call \textit{almost coherent sequence}:
\begin{definition}\label{Def:almostcoherentesequence}
     An \emph{almost coherent sequence} is a sequence $$\vec{U}=\l U(\alpha,\beta)\mid \alpha\leq \kappa, \beta<o^{\vec{U}}(\alpha)\r$$
     such that:
     \begin{enumerate}
         \item $U(\alpha,0)$ is an $\alpha$-complete (non-necessarily normal) ultrafilter over $\alpha$.
         \item for $\alpha\leq\kappa$ and $0<\beta<o^{\vec{U}}(\alpha)$, $U(\alpha,\beta)$ is a normal measure on $\alpha$
     \item for every $\l \alpha,\beta\r\in\dom(\vec{U})$, $$j_{U(\alpha,\beta)}(\vec{U})\restriction \alpha+1=\vec{U}\restriction (\alpha,\beta)$$
     where $$\vec{W}\restriction\alpha+1=\l W(\gamma,\beta)\mid \gamma\leq \alpha, \ \beta<o^{\vec{W}}(\gamma)\r$$
     and $$\vec{W}\restriction (\alpha,\beta)=\vec{W}\restriction\alpha^{\smallfrown}\l W(\alpha,\gamma)\mid \gamma<\beta\r$$
     whenever $\vec{W}$ is an almost coherent sequence
      \end{enumerate} 
 \end{definition}
 In the above definition if $\beta=0$ then $j_{U(\alpha,0)}(\vec{U})\restriction\alpha+1=\vec{U}\restriction\alpha$. That is, we require that there are no measures on $\alpha$ in $j_{U(\alpha,0)}(\vec{U})$.
 

\begin{corollary}\label{Cor:almostCohereSequence}
    Let $\langle V_\alpha\mid \alpha<\omega_1\rangle$ a $\triangleleft$-increasing sequence be such that $V_\alpha$ is normal for all $\alpha>0.$ There is an almost coherent sequence $\vec{U}$ such that  $\vec{U}(\kappa,\alpha)=V_\alpha$ for all $\alpha<\omega_1$.  
\end{corollary}
Note that $\{V^{\text{nor}}_0\}\cup \{V_i\mid i<\gamma\}$ are all distinct ultrafilters, and therefore we can make the coherent sequence special. The proof of the following proposition is a straightforward verification:
\begin{proposition}
    Suppose that $\vec{U}$ is a generalized coherent sequence of length $\gamma$ with a top cardinal $\kappa$, then for each $\alpha\leq \kappa$ and $i\leq o^{\vec{U}}(\alpha)$, $\l U(\beta,r)\mid \beta<\alpha, r<\min(o^{\vec{U}}(\beta),i)\r^{\smallfrown}\l U(\alpha,j)\mid j<i\r$ is a generalized coherent sequence of length $i$ with a top cardinal $\alpha$.
\end{proposition}
We denote the above generalized coherent sequence by $\vec{U}\restriction (\alpha,i)$.

\section{Non-normal Magidor forcing with a coherent sequence}\label{Sec: non-normalMagidor} 
In this section we  generalize the presentation of Magidor forcing due to Mitchell \cite{MITCHELLHowWeak} (see also \cite{Gitik-handbook}) which has been also studied by the first author and Gitik in a series of papers \cite{TomMaster,partOne,TomMotiII}.
\begin{proposition}
    Let $\vec{U}$ be a generalized coherent sequence with a top cardinal $\kappa$ and let $\l A_i\mid i<o^{\vec{U}}(\kappa)\r$ be a sequence of sets such that $A_i\in U(\kappa,i)$. Then for every $i<\kappa$,  
    $$\{\nu\in A_i\mid o^{\vec{U}}(\nu)=i, \ \forall j<i\   A_j\cap \pi(\nu)\in U(\nu,j)\}\in U(\kappa,i).$$
\end{proposition}
\begin{proof}
    For every $i<o^{\vec{U}}(\kappa)$, and $j<i$, $$j_{U(\kappa,i)}(A_j)\cap j_{U(\kappa,i)}(\pi)([\mathrm{id}]_{U(\kappa,i)})=A_j$$ and by coherency, $ U(\kappa,j)=j_{U(\kappa,i)}(\vec{U})([\mathrm{id}]_{U(\kappa,i)},j).$ It follows that 
    $$M_{U(\kappa,i)}\models \forall j<i, \ j_{U(\kappa,i)}(A_j)\cap j_{U(\kappa,i)}(\pi)([\mathrm{id}]_{U(\kappa,i)})\in j_{U(\kappa,i)}(\vec{U})([\mathrm{id}]_{U(\kappa,i)},j)$$
\end{proof}
\begin{notation}
    A \emph{basic pair} is a pair $(\alpha,A)$ where $$\textstyle A\in \bigcap_{i<o^{\vec{U}}(\alpha)}U(\alpha,i)=:\bigcap\vec{U}(\alpha).$$ By convention, if $o^{\vec{U}}(\alpha)=0$, then the sequence is empty, so universal statements about it will be vacuously true. For a basic pair $t=(\alpha, A)$, we denote $\kappa(t):=\alpha$ and $A(t):=A$.
\end{notation}
\begin{definition}
    Let $\vec{U}$ be a generalized coherent sequence of length $\gamma$ with a top cardinal $\kappa$. We define the condition of the \emph{non-normal-Magidor forcing} $\mathbb{M}[\vec{U}]$ as the poset consisting of conditions $\l t_1,...,t_n,\l \kappa,A\r\r$ such that:
    \begin{enumerate}
        \item each $t_i$ is a basic pair.
        \item  For each $\alpha\in A(t_i)\cup\{\kappa(t_i)\}, \  \pi(\alpha)>\kappa(t_{i-1})$.  
    \end{enumerate}
\end{definition}

\begin{definition}
    The order for $\mathbb{M}[\vec{U}]$ is defined by $$\l t_1,...,t_n,\l \kappa,A\r\r\leq \l s_1,...,s_m,\l \kappa,B\r\r$$ whenever there are indices $i_{0}:=0<1\leq i_1<...<i_n\leq m=:i_{n+1}$ such that  for each $1\leq r\leq n+1$
    \begin{enumerate}
        \item  $\kappa(s_{i_r})=\kappa(t_r)$,  $A(s_{i_r})\subseteq A(t_r)\setminus \pi^{-1}[\kappa(s_{i_r-1})+1]$.
        
        \item If $i_{r-1}<j<i_r$, then 
        \begin{enumerate}
            \item $\kappa(s_j)\in A(t_r)$.
            \item $o^{\vec{U}}(\kappa(s_j))<o^{\vec{U}}(\kappa(t_r))$.
            \item $A(s_j)\subseteq (A(t_r)\cap \pi(\kappa(s_j)))\setminus \pi^{-1}[\kappa(s_{j-1})+1]$. 
        \end{enumerate}
    \end{enumerate}
    In case $n=m$ (and therefore $i_r=r$) we write $p\leq^* q$.
\end{definition}
\begin{proposition}
    The order $\leq$ on $\mathbb{M}[\vec{U}]$ is transitive
\end{proposition}
\begin{proof}
    Suppose $$\l t_1,...,t_n,\l \kappa,A\r\r\leq \l s_1,...,s_m,\l \kappa,B\r\r\leq \l z_1,...,z_k,\l \kappa,C\r\r.$$
    By definition there are $$1\leq i_1<...\leq i_n\leq m\text{ and }1\leq j_1<j_2<...<j_m\leq k$$ witnessing the left and right inequalities (resp.). Define $l_r=j_{i_r}$. then $1\leq l_1<...<l_n\leq k$. Let us prove that $(1),(2a)-(2c)$ hold:
    \begin{enumerate}
        \item First, $\kappa(z_{l_r})=\kappa(z_{j_{i_r}})=\kappa(s_{i_r
        })=\kappa(t_r)$. Moreover, $$A(z_{l_r})\subseteq A(s_{i_r})\cap \pi^{-1}[\kappa(z_{l_r-1})+1]\subseteq A(t_r)\setminus \pi^{-1}[\kappa(z_{l_r-1})+1].$$
        \item Suppose that $l_{r-1}<j<l_r$. and let us split into cases:
        \begin{itemize}
            \item [Case 1:] There is $1\leq w\leq m$ such that $j=j_w$, in which case, $i_{r-1}<w<i_{r}$ and therefore
            \begin{enumerate}
                \item $\kappa(z_j)=\kappa(s_w)\in A(t_r)$.
                \item $o^{\vec{U}}(\kappa(z_j))=o^{\vec{U}}(\kappa(s_w))<o^{\vec{U}}(\kappa(t_r))$.
                \item $A(z_j)\subseteq A(s_w)\setminus \pi^{-1}[\kappa(z_{j-1})+1]\subseteq$\\ $\subseteq A(t_r)\cap \pi(\kappa(z_j))\setminus \pi^{-1}[\kappa(z_{j-1})+1]$.
            \end{enumerate}
            
                \item [Case 2.]  $j_{w-1}<j<j_w$, in which case, $i_{r-1}<w\leq i_{r}$ and therefore
             \begin{enumerate}
                \item $\kappa(z_j)\in A(s_w)\subseteq A(t_r)$.
                \item $o^{\vec{U}}(\kappa(z_j))<o^{\vec{U}}(\kappa(s_w))\leq o^{\vec{U}}(\kappa(t_r))$.
                \item $A(z_j)\subseteq A(s_w)\cap \pi(\kappa(z_j))\setminus \pi^{-1}[\kappa(z_{j-1})+1]\subseteq$\\ $\subseteq A(t_r)\cap \pi(\kappa(z_j))\setminus \pi^{-1}[\kappa(z_{j-1})+1]$.
            \end{enumerate}
        \end{itemize}
\end{enumerate}
\end{proof}
\begin{remark}
By condition $(2b)$ of the order on $\mathbb{M}[\vec{U}]$, given a set $A\in \cap \vec{U}(\alpha)$, we may assume always that $A=\biguplus_{j<o^{\vec{U}}(\alpha)} A^{(j)}$ , where $$A^{(j)}=\{\nu\in A\mid o^{\vec{U}}(\nu)=j\}.$$
\end{remark}
\begin{notation}
    Given $p=\l t_1,...,t_n,\l\kappa,A\r\r\in \mathbb{M}[\vec{U}]$, let \begin{enumerate}
        \item $l(p)=n$.
        \item $t_i(p)=t_i$, 
        \item $t_{n+1}(p)=\l \kappa,A\r$ and in particular $A(t_{n+1}(p))=A$.
        \item $p\restriction i+1=\l t_1,...,t_i\r$.
        \item $p\restriction [i+1,n+1]=\l t_{i+1},...,t_n,t_{n+1}\r$.
    \end{enumerate}
\end{notation}
The following are straightforward:
\begin{proposition}
    $\mathbb{M}[\vec{U}]$ is $\kappa$-centered and therefore $\kappa^+$-cc.
\end{proposition}
\begin{proposition}
    For $p\in\mathbb{M}[\vec{U}]$, $(\mathbb{M}[\vec{U}]/p,\leq^*)$ is $\kappa(t_1(p))$-directed-closed.
\end{proposition}
\begin{proposition}
    Given $p\in\mathbb{M}[\vec{U}]$ and $1\leq i\leq l(p)$
    $$\mathbb{M}[\vec{U}]/p\simeq \big(\mathbb{M}[\vec{U}\restriction \pi(t_i(p))]/p\restriction i+1\big)\times \big(\mathbb{M}[\vec{U}]/p\restriction[ i+1,n+1]\big)$$
\end{proposition}
\begin{definition}
    Let $p=\l t_1,..,t_n,t_{n+1}\r\in\mathbb{M}[\vec{U}]$ and $\alpha\in A(t_i)$, for some $1\leq i\leq n+1$, define
    $$p^{\smallfrown}\l\alpha\r=\l t_1,...,t_{i-1},\l \alpha,A(t_i)\cap \pi(\alpha)\r, \l \kappa(t_i),A(t_i)\setminus \pi^{-1}[\alpha+1]\r,t_{i+1},...,t_{n+1}\r.$$
    We define recursively $p^{\smallfrown}\l\alpha_1,..,\alpha_n\r=(p^{\smallfrown}\l\alpha_1,...,\alpha_{n-1}\r)^{\smallfrown}\l\alpha_n\r$.
\end{definition}
The proof for the Prikry property and the strong Prikry property will be given in the next section for the non-normal Radin forcing, but the proof is completely analogous and therefore is omitted.
\begin{definition}
    A tree $T\subseteq [\kappa]^{<\omega}$ of height $n$, consisting of $\pi$-increasing sequence is called $\vec{U}$-fat if for every $t\in T$, such that $|t|<n$, there is $i<o^{\vec{U}}(\kappa)$ such that $\Succ_T(t)=\{\alpha\mid t^{\smallfrown}\alpha\in T\}\in U(\kappa,i)$. Suppose that $\vec{V}=\l\vec{v}_1,...,\vec{v}_n\r$ is a sequence of generalized coherent sequences, a sequence of trees $\vec{T}=\l T_1,...,T_n\r$ is called $\vec{v}$-fat if for each $1\leq i\leq n$, $T_i$ is $\vec{v}_i$-fat. 
\end{definition}
If for every $1\leq i\leq n$, the coherent sequences $\vec{v}_i$ above happens to be the coherent sequence $\vec{U}\restriction \pi(\kappa(t_i(p)))$ for some given condition $p$ of length $n$, then, we say that $T$ is fat below $p$.
For a tree $T$ of height $n$, we denote the set of maximal branches in $T$ by $mb(T)=\{t\in T\mid |t|=n\}$.
\begin{theorem}[The strong Prikry property]
    Let $p\in \mathbb{M}[\vec{U}]$ be any condition and $D\subseteq \mathbb{M}[\vec{U}]$ be dense open. Then there $p\leq^* p^*$ and and a sequence $\vec{T}=\l T_1,...,T_{l(p)+1}\r$ fat below $p^*$ such that for every sequence of branches $\l b_1,...,b_{l(p)+1}\r\in \prod_{1\leq i\leq l(p)+1}mb(T_i)$, $$p^{*\smallfrown}b_1^{\smallfrown}b_2...^\smallfrown b_n\in D$$   
\end{theorem}
\begin{corollary}
    $\mathbb{M}[\vec{U}]$ preserves all cardinals.
\end{corollary}
Other properties of the classical Magidor forcing $\mathbb{M}[\vec{U}]$ can be generalized to our non-normal version -- this shall not be presented here. For more about these properties see \cite{TomMaster,partOne, TomMotiII}.

One important difference between this forcing and the usual normal Magidor forcing is that the generic sequence is not closed anymore:
\begin{definition}
    Let $G\subseteq \mathbb{M}[\vec{U}]$ be $V$-generic. The generic object added by $G$ is
    $$C_G:=\{\alpha\mid \exists p\in G \ \exists 1\leq i\leq l(p), \ \kappa(t_i(p))=\alpha\}.$$
\end{definition}
It is not hard to check that for every $A\in \bigcap\vec{U}(\kappa)$, $C_G\subseteq^* A$  and that $V[G]=V[C_G]$. However, this sequence is not normal:
\begin{proposition}
    Let $\vec{U}$ be a generalized coherent sequence coherent sequence of length $\gamma$ with a top cardinal $\kappa$, and let $G$ be generic for $\mathbb{M}[\vec{U}]$. For $0<\alpha<\gamma$,  
    $U(\kappa,\alpha)$ is normal if and only if there is $\xi<\kappa$ such that for every $\rho\in C_G\setminus\xi$ with $o^{\vec{U}}(\rho)=\alpha$, $\sup(C_G\cap \rho)=\rho$.
\end{proposition}
\begin{proof}
    Suppose that $U(\kappa,\alpha)$ is normal, then $[\pi]_{U(\kappa,\alpha)}=[id]_{U(\kappa,\alpha)}$. Hence there is $A\in\cap\vec{U}(\kappa)$ such that for every $\rho\in A$, if $o^{\vec{U}}(\rho)=\alpha$ then $\pi(\rho)=\rho$. Hence there is $\xi<\kappa$ such that $C_G\setminus \xi\subseteq A$. Now suppose that $\rho\notin C_G\setminus \xi$ and $o^{\vec{U}}(\rho)=\alpha$ and let $p\in G$ be a condition such that $\kappa(t_i(p))=\rho$. The for every $\delta'<\pi(\rho)=\rho$, and for every $p\leq q$, there is $j$ such that $\kappa(t_j(q))=\rho$ and therefore there is  $\gamma\in A(t_j(q))\setminus \delta'$. Now $q^\smallfrown\l \delta'\r$ is a condition forcing that $\sup(C_G\cap \rho)\geq\delta'$. By density, there is such a condition in $G$ and since $\delta'<\rho$ was arbitrary, $\sup(C_G\cap\rho)=\rho$. If $U(\kappa,\alpha)$ is non-normal, then there is $\xi<\kappa$ such that for every $\rho\in C_G\setminus\xi$ with $o^{\vec{U}}(\rho)=\alpha$, $\pi(\rho)<\rho$. Now let $p\in C_G$ be any condition with $\rho=\kappa(t_i(p))$ for some $1\leq i\leq p$, then for every $j>i$, and every $\alpha\in A(t_j(p))$, $\pi(\alpha)>\kappa(t_{j-1})\geq\rho$, $p\Vdash C_G\cap \rho=C_G\cap \pi(\rho)$. Thus $\sup(C_G\cap\rho)\leq \pi(\rho)<\rho$.
\end{proof}
 \begin{lemma}
     Let $\alpha$ be a regular cardinal in $V$. If $p\Vdash \alpha\notin \mathrm{acc}(C_G)$, then $p\Vdash \cf(\alpha)=\alpha$. In particular all cofinalities below $\delta_0:=\min\{\nu\mid o^{\vec{U}}(\nu)>0\}$ are preserved.
 \end{lemma} 
The non-normal Magidor-Radin forcing appeared implicitly in the work of Merimovich \cite{CarmiMagidorRadin,CarmiRadin}. The analogy is the following: 
the tree-Prikry forcing appears as a projection of both the usual Gitik-Magidor Extender-Based Prikry forcing of \cite{Git-Mag} and its more modern presentation due to Merimovich \cite{Mer}. Next, we will show that the non-normal Magidor forcing $\mathbb{M}[\vec{U}]$ appears as a projection of Merimovich's Extender-Based Magidor/Radin forcing \cite{CarmiMagidorRadin}.

\smallskip

Given a Mitchell increasing sequence of (short) extenders $$\bar{E}=\l E_\xi\mid \xi<\gamma\r$$ a condition in the forcing $\mathbb{P}_{\bar{E}}$ has the form $$\l \l f_0,A_0,\bar{e}_0\r,...,\l f_n,A_n,\bar{e}_n\r,\l f,A,\bar{E}\r\r,$$ where
\begin{enumerate}
\item $\bar{e}^i=\l \bar{e}^i_j\mid j<o(\bar{e}^i)\r$ is an extender sequence with critical point $\kappa_i$.
    \item $\dom(f_i)\in P_{\kappa_i^+}(\mathfrak{D}_i)$,  $f_i:\dom(f_i)\rightarrow\mathfrak{R}_i^{<\omega}$, where $\mathfrak{D}_i$ is the set of all possible coordinates for the extender sequence $e^i$ and $\mathfrak{R}_i$ is the set of ranges for $f_i$ (which consists of extender sequences\footnote{An extender sequence is a sequence of the form $\xi=\l\rho\r^\smallfrown \l e_i\mid i<j\r$ where $e_i$ is a Mitchell increasing sequence of extenders. where $crit(e_0)\leq \rho<j_{e_0}(\crit(e_0))$.We denote by $\xi_0=\rho$ and $\bar{e}(\xi)=\l e_i\mid i<j\r$.} below $crit(e^i)$).
    \item $A_i$ is a $\dom(f_i)$-tree; namely, for every $\vec{\nu}\in A_i$, $$\Succ_{A_i}(\vec{\nu})\in\bigcap_{j<o(e^i)}e^i_j(\dom(f_i)).$$
    
\end{enumerate}
We refer the reader to Merimovich's paper \cite{CarmiMagidorRadin} for a complete account of the  
Magidor-Radin extender-based forcing. 

Recall that if $E$ is a $(\kappa,\lambda)$-extender and $\alpha<\lambda$, then $k_\alpha:M_{E_\alpha}\rightarrow M_E$ is an elementary embedding defined by $k_\alpha([f]_{E_\alpha})=j_E(f)(\alpha)$. The next proposition expresses that from a Mitchell-increasing sequence of extenders, one can derive many Mitchell-increasing sequences of ultrafilters.
\begin{proposition}
    Assume GCH. Suppose that $E$ is a $(\kappa,\lambda)$-extender on $\kappa$, and $U\in M_E$ is an ultrafilter on $\kappa$ such that $U=j_E(f)(\xi_1,...,\xi_n)$, then for any $\alpha$ with $\kappa,\xi_1,...,\xi_n\in \text{rng}(k_\alpha)$, $U\in M_{E_\alpha}$.
\end{proposition}
\begin{proof}
    By the assumption, we have that $\crit(k_\alpha)\geq (\kappa^{++})^{M_{E_\alpha}}$. Let $\rho_1,...,\rho_n$ be preimages of $\kappa,\xi_1,..,\xi_n$ under $k_\alpha$ respectively. We have $$U'=j_{E_\alpha}(f)(\rho_1,...,\rho_n)\in M_{E_\alpha}.$$ Let us prove that $U'=U$. Indeed, for every $X\subseteq \kappa$ ($P(\kappa)$ is the same in all the models), we have that 
    $X\in U'$ if and only if $k_\alpha(X)\in U$. But we have that $k_\alpha(X)=X$ as the critical point of $k_\alpha$ is above $\kappa$.
\end{proof}
Suppose that we are given for every $i<o(\bar{E})$, $\alpha_i<j_{E_0}(\kappa)$ such that $\l E_i(\alpha_i)\mid \ell\leq i<o(\bar{E})\r$ is Mitchell increasing. Let $\vec{U}$ be the coherent sequence derived from $\l E_i(\alpha_i)\mid \ell\leq i<o(\bar{E})\r$, and let us prove that $\mathbb{M}[\vec{U}]$ is a projection of $\mathbb{P}_{\bar{E}}$. 
\begin{theorem}\label{Thm:Projection of extender-based Radin}
    Let $\bar{E}$ be a Mitchell increasing sequence of extenders with $o(\bar{E})<\kappa=\crit(E)$. For every $\ell<o(\bar{E})$, let $\kappa\leq \alpha_{\ell}<j_{E_0}(\kappa)$ be an ordinal such that $\l E_i(\alpha_i)\mid i<o(\bar{E})\r$ is $\triangleleft$-increasing and let $\vec{U}$ be the generalized coherent sequence derived from $\l E_i(\alpha_i)\mid i<o(\bar{E})\r$. Then $\mathbb{M}[\vec{U}]$ is a projection of $\mathbb{P}_{\bar{E}}$ above a certain condition.
\end{theorem}
\begin{proof} 
     Consider the condition $p_*=\l\l f_*,A_*,\bar{E}\r\r\in\mathbb{P}_{\bar{E}}$, where $A_*$ is a $d$-tree such that $\dom(f_*)=\{\bar{\kappa}\}\cup\{\bar{\alpha}_i\mid i<o(\bar{E})\}$ and for each $\dom(f_*)$-object $\nu$ in  $A_*$, and $i<o(\bar{E})$ with $o(\nu)=i$,   $\dom(\nu)=\dom(f_*)$ and $$(\star) \ \ o(\nu(\bar{\alpha}_i))=o^{\vec{U}}(\nu(\bar{\alpha}_i)_0)\text{ and }\vec{U}(\nu(\bar{\alpha}_i)_0,j)=\bar{e}_j(\nu)(\nu(\bar{\alpha}_i)_0).$$ Note that such a set $A_*$ exists since for every $i<o(\bar{E})$, $$j_{E_i}(\vec{U})(\alpha_i)=k(j_{E_i(\bar{\alpha}_i)}(\vec{U})([id]_{E_i(\bar{\alpha}_i)}))$$ 
     Since $\vec{U}$ is coherent and $U(\kappa,i)=E_i(\bar{\alpha}_i)$,  we have that $$j_{E_i(\bar{\alpha}_i)}(\vec{U})([id]_{E_i(\bar{\alpha}_i)})=\l U(\kappa,j)\mid j<i\r=\l E_j(\bar{\alpha}_j)\mid j<i\r$$ and since $i,\kappa^{+}<\crit(k)$, we conclude that $$\l U(\kappa,j)\mid j<i\r=k(\l U(\kappa,j)\mid j<i\r)=j_{E_i}(\vec{U})(\alpha_i).$$     Finally,  since $\kappa\leq \alpha_i<j_{E_0}(\kappa)$ we have
     $$R_i(\bar{\alpha}_i)=\langle\alpha_i\rangle^\smallfrown\l E_j\mid j<i\r$$ (see Definition 4.4 in \cite{CarmiMagidorRadin}). 
     Thus,
     $$j_{E_i}(\vec{U})(mc_i(\bar{\alpha}_i)(j_{E_i}(\bar{\alpha}_i))_0)=j_{E_i}(\bar{U})(\alpha_i)=\l U(\kappa,j)\mid j<i\r=$$
     $$=\l E_j(\bar{\alpha}_j)\mid j<i\r=\l\bar{e}_j(mc_i(\bar{\alpha}_i)(j_{E_i}(\bar{\alpha}_i)))(\bar{\alpha}_j)\mid j<i\r.$$
     Note that by squashing every $\nu\in A_*$ by some $\mu$ (i.e, taking $\nu\circ \mu^{-1}$), condition $(\star)$ does not changes. 
     
     Given a $d$-tree $B\subseteq Ob(d)^{<\omega}$, and $\bar{\gamma}=\l \tau,e_0,e_1,\cdots e_\alpha\cdots\mid \alpha<o(\bar{\gamma})\r\in d$ (say $\kappa_0=\crit(e_0)$) we define $$(B\restriction \bar{\gamma})^*=\{\nu(\bar{\gamma})_0\mid \nu\in Ob(d), \  \forall \vec{\xi}\in B\cap V_{\nu(\kappa_0)}(\nu\in\Succ_B(\vec{\xi}))\}.$$
     Then $$(B\restriction\bar{\gamma})^*\in \bigcap_{i<o(\bar{\gamma})} e_i(\bar{\gamma}).$$
     Indeed, if $i<o(\bar{\gamma})$ and let $\vec{\xi}\in j_{e_i}(B)\cap V_{\kappa_0}=B$, then since $B$ is a $d$-fat, $mc_i(d)\in j_{e_i}(\Succ_B(\vec{\xi}))=\Succ_{j_{e_i}(B)}(\vec{\xi})$. Thus, $mc_i(d)(j_{e_i}(\bar{\gamma}))_0\in j_{e_i}(B\restriction \bar{\gamma})^*)$ 
     
     Since every condition $p\in\mathbb{P}_{\bar{E}}/p_*$ is of the form $p\leq^* p_*^{\smallfrown}\l\nu_1,...,\nu_k\r$, we can define $\pi(p)$ recursively on the length of $p$.
     First define $\pi(p_*)=\l\l\kappa,(A\restriction\overline{\alpha})^*\r\r$. Note that $(A\restriction\overline{\alpha})^*\in \bigcap_{i<o^{\vec{U}}(\kappa)}U(\kappa,i)$. Now given 
     a condition $$p=\l \l f^0,A_0,\bar{e}^0\r,...,\l f^n,A_n,\bar{e}^n\r,\l f^{n+1},A,\bar{E}\r\r,\in \mathbb{P}_{\bar{E}}/p_*,$$
suppose we have already defined
     $$\pi(p)=\l \vec{\alpha}_1, t_1,...,\vec{\alpha}_n,t_n,\vec{\alpha}_{n+1},t_{n+1}\r$$ such that for every $1\leq i\leq n$, 
     \begin{enumerate}
         \item $\overline{\kappa(t_i)}\in \dom(f^{i})\cap (crit(\bar{e}^i),j_{\bar{e}^i_0}(crit(\bar{e}^i))$. We assume $\kappa(t_{n+1})=\alpha$.
         \item $\vec{\alpha}_i=\l f^i_k(\kappa(t_i))_0,...,f^i_r(\kappa)_0\r$, where $k$ is the minimal such that for every $l\geq k$, $o(f^i_l(\kappa(t_i)))=0$. In particular $\vec{\alpha}_i$ is a sequence of ordinal of order $0$,
         \item $A(t_i)=(A_i\restriction \overline{\kappa(t_i)})^*$. 
         
         \end{enumerate}
         Note that specifying $\kappa(t_i)$ completely determines $\pi(p)$ from $p$. In particular, 
        any given $q\leq^* p$, has to be defined as a direct extension of $\pi(p)$ by shrinking for each $i\leq n+1$,  $A(t_i)=(A_i\restriction \overline{\kappa(t_i)})^*$ to $(A^{q}_i\restriction \overline{\kappa(t_i)})^*$. In that case $\pi(q)\leq^*\pi(p)$, and for every direct extension $x\leq^* \pi(p)$ there is a direct extension $q\leq^* p$ such that $\pi(q)\leq^*x$.
        
        Now given $\nu\in A_r$ for some $1\leq r\leq n+1$, 
        
        \begin{itemize}
            \item [\underline{Case 1}] if  $o(\nu(\kappa(t_r))=0$, then for every $\bar{\gamma}\in \dom(f^r)$, $o(\nu(\bar{\gamma}))=0$ (see Definition 4.3 item (5) in \cite{CarmiMagidorRadin}) and therefore in $p^\smallfrown\nu$ we only append points of order $0$ to end extend the sequences of the Cohen function $f^r$ (see Definition 4.4 in \cite{CarmiMagidorRadin}).  So in this case $$\pi(p^\smallfrown\nu)=\l \vec{\alpha}_1, t_1,...,\vec{\alpha}'_r,t_r,...\vec{\alpha}_{n+1},t_{n+1}\r$$
        where $\vec{\alpha}'_r=\vec{\alpha}_r^\smallfrown\nu(\overline{\kappa(t_r)})_0$. Note that by definition, $\nu(\overline{\kappa(t_r)})_0\in (A_r\restriction \overline{\kappa(r_t)})^*=A(t_i)$ and therefore $\pi(p^\smallfrown\nu)$ is a legitimate one-point extension of $\pi(p)$.
        \item [\underline{Case 2}:] if $o(\nu(\overline{\kappa(t_r)}))>0$, We define $$\pi(p^{\smallfrown}\nu)=\l\vec{\beta}_1, s_1,...\vec{\beta}_{n+1},s_{n+1},\vec{\beta}_{n+1},s_{n+1}\r$$  by specifying $$\kappa(s_i)=\begin{cases}\kappa(t_i)& i<r\\\nu(\kappa(t_r))_0 & i=r\\
    \kappa(t_{i-1}) &r<i\leq n+1\end{cases}$$  
    By definition, $\nu(\overline{\kappa(t_r)})_0\in (A_r\restriction \overline{\kappa(r_t)})^*=A(t_r)$, and therefore $\nu(\kappa(t_r))_0$ is a legitimate ordinal to be added to $\pi(p)$.
    Denote by $\bar{e}=\bar{e}(\nu)$ we note that by $(\star)$, $$U(\nu(\kappa(t_r)),j)=\bar{e}_j(\overline{\kappa(t_r)}).$$
    Therefore $A(s_r)=((A\downarrow\nu)\restriction \overline{\nu(\kappa(t_r))})^*\in\bigcap_{j<o^{\vec{U}}(\nu(\kappa(r_t)))}U(\nu(\kappa(t_r)),j)$. 
    
    \end{itemize}
 Also, note that every one-point extension $\pi(p)^\smallfrown\rho$ of $\pi(p)$ using an order $0$ ordinal there is some $\nu$ with the same order (this is due to $(\star)$) such that $\pi(p^\smallfrown\nu)=\pi(p)^\smallfrown\rho$.
      Since every extension of $\pi(p)$ is of the form $q\leq^* \pi(p)^\smallfrown\l\rho_1,..,\rho_n\r$, we conclude that $\pi$ is a projection.
\end{proof}
\section{Non-normal Radin forcing}\label{SectionNonNormalRadin}
Suppose that $\kappa$ is a $\mathcal{P}_2\kappa$-hypermeasurable cardinal as witnessed by an elementary embedding $j:V\rightarrow M$. 
Let $\sigma <j(\kappa)$\label{pi} and $\pi:\kappa\rightarrow\kappa$ be a function such that $j(\pi)(\sigma)=\kappa$. Following Cummings and Woodin's \cite{Cummings} we derive a sequence of ultrafilters as follows: $u^j(0):=\l \sigma\r$ and for each $\xi\geq 1$
$$u^j(\xi):=\{X\s V_\kappa\mid u^j\restriction\xi\in j(X)\}.$$\label{measuresequence}
The construction of the $u^j$ is continued until reaching $\xi$ such that $u^j\restriction \xi\notin M.$ The least ordinal $\xi$ where the construction halts will be denoted $\ell(u^j).$

\smallskip

Note that both $u^j(1)$ and $u^j(2)$ are $\kappa$-complete measure and that $u^j(2)$ concentrates on pairs $\l \beta, w_\beta\r$ where $w_\beta$ is a $\pi(\beta)$-complete ultrafilter over $V_{\pi(\beta)}$. Unlike in the usual construction of Radin forcing \cite{Rad} (see \cite[\S5.1]{Gitik-handbook}) our $u(1)$ here is a non-normal measure on $V_\kappa$.

\begin{notation}
    For a sequence $u$ as before $\sigma_u$ denotes the  ordinal in $u(0).$
\end{notation}
\begin{definition}[Measure sequences and measure one sets]\hfill

\begin{enumerate}
    \item $\mathcal{MS}_0:=\{u\in V_{\kappa+2}\mid \exists j\colon V\rightarrow M\;\exists\alpha\leq \ell(u^j)\; u=u^j\restriction\alpha\}$;
    \item $\mathcal{MS}_{n+1}:=\{u\in \mathcal{MS}_n\mid \forall \xi\in [1,\ell(u))\, \mathcal{MS}_n\cap V_{j_u(\pi)(\sigma_u)}\in u(\xi)\}$.\footnote{Here $j_u$ stands for an embedding witnessing $u\in\mathcal{MS}_0.$}
\end{enumerate}
    The collection of \emph{measure sequences} $\mathcal{MS}$ is defined as $\bigcap_{n<\omega}\mathcal{MS}_n$.

    Given $u\in \mathcal{MS}$ denote by $\mathscr{F}(u)$ the filter associate to $u$; namely,
    $$\mathscr{F}(u):=\begin{cases}
        \{\emptyset\}, & \text{if $\ell(u)=1$};\\
        \bigcap_{1\leq \xi<\ell(u)} u(\xi), & \text{if $\ell(u)\geq 2$.}
    \end{cases}$$
    For $A\in \mathscr{F}(u)$ and $1\leq \xi<\ell(u)$ we will denote $(A)_\xi:=\{w\in A\mid \ell(w)=\xi\}$.\footnote{Note that $(A)_\xi \in \mathcal{F}(u\restriction\xi)$.}
\end{definition}
The next lemma due to Cummings \cite{Cummings} shows that one can derive long measure sequences from $\mathcal{P}_2\kappa$-hypermeasurable embeddings:
\begin{lemma}
    Let $\kappa$ be a $\mathcal{P}_2\kappa$-hypermeasurable cardinal and $j\colon V\rightarrow M$ a witnessing embedding. Then, $\ell(u^j)\geq (2^\kappa)^+$ and $u^j\restriction\alpha\in \mathcal{MS}$ for $\alpha<\ell(u^j).$
\end{lemma}
In what follows 
$u$ will be the truncation  of the measure sequence $u^j$ derived from $j$ using some  $\sigma<j(\kappa)$ as a seed; to wit, $u=u^j\restriction\alpha$ for some $\alpha<\ell(u^j).$ Likewise we will fix a function $\pi\colon \kappa\rightarrow \kappa$ such that $j(\pi)(\sigma)=\kappa$.

\smallskip

We define an ordering between measure sequences as follows:
\begin{definition}\label{OrderBetweenMeasures}
    Given $v, w\in \mathcal{MS}\cap V_\kappa$ write $v\prec w$ whenever $v\in V_{\kappa_w}$. Here we denoted $\kappa_w:=\pi(\sigma_w)$.
\end{definition}
\begin{remark}
    Since $w\in V_\kappa$ it follows that $\sigma_w<\kappa$ and as a result $\pi(\sigma_w)$ is well-defined. Also, observe that $v\prec w$ entails  $\sigma_u<\pi(\sigma_v)$.
\end{remark}

We are now in conditions to define the non-normal Radin forcing:

\begin{definition}\label{TreeRadin}
    The \emph{Radin forcing} $\mathbb{R}_u$ consists of finite sequences 
    $$p=\l \l u^p_0, A^p_0\r,...,\l u^p_{\ell(p)-1}, A^p_{\ell(p)-1}\r,\l u^p_{\ell(p)},A^p_{\ell(p)}\r\r $$
    where 
    \begin{enumerate}
        \item $u^p_{\ell(p)}=u$ and $u^p_i\in \mathcal{MS}\cap V_\kappa$ for all $i< \ell(p)$,
        \item $A^p_i\in \mathscr{F}(u^p_i)$ for all $i\leq \ell(p)$,
        \item $\langle u^p_i\mid i<\ell(p)\rangle$ is $\prec$-increasing, 
        \item and $u^p_{i}\prec v$ for all $v\in A^p_{i+1}$ and $i< \ell(p)$.
    \end{enumerate}
    When $p$ is clear from the context we will tend to suppress the superindex $p$.
    
    Given  $p,q\in \mathbb{R}_u$ write $p\leq^* q$ whenever $\ell(p)=\ell(q)$, $u^p_i=u^q_i$ and $A^p_i\s A^q_i$. 
\end{definition}
The minimal one-point extensions of a condition are given as follows:
\begin{definition}
Let $p=\l \l u_0, A_0\r,...,\l u_{n-1}, A_{n-1}\r,\l u,A_n\r\r \in\mathbb{R}_{u}$, $i\leq n$ and $v\in A_i$. We define the \emph{one-point extension of $p$ by $v$}, $p\cat v$, as follows:
$$p\cat v:=\l \l u_0, A_0\r,...,\l v,A_i\downarrow v\r,\l u_i,(A_i)_v\r,...,\l u_{n-1},A_n\r,\l u,A\r\r, $$
where $A_i\downarrow v:=\{w\in A_i\cap V_{\kappa_v}\mid \ell(w)<\ell(v)\}$ and $(A_i)_v:=\{w\in A_i\mid v\prec w\}.$

Given a (non-necessarily $\prec$-increasing) sequence  $\langle v_i\mid i\leq k\rangle$ one defines $p\cat \langle v_i\mid i<k\rangle$ by recursion as $(p\cat \langle v_i\mid i<k\rangle)\cat v_k.$
\end{definition}
For certain $v\in \mathcal{MS}$ it is plausible that $p\cat v$ is not a well-defined condition. Let us call a condition $p\in \mathbb{R}_u$ \emph{pruned} if $p\cat \langle v_i\mid i\leq k\rangle$ is a condition for all finite sequences $\langle v_i\mid i\leq k\rangle$ in the measure one sets of $p$. It is not hard to show  that the set of pruned conditions is $\leq^*$-dense in $\mathbb{R}_u$. Thus, we do not loss any generality by assuming that all of our conditions in $\mathbb{R}_u$ are pruned. 
\begin{definition}[The forcing ordering]
    For two conditions $p,q\in \mathbb{R}_u$ we write $p\leq q$ if there is $\langle v_i\mid i\leq k\rangle\in \prod_{i\leq k} A^q_{j_i}$ such that $p\leq^* q\cat \langle v_i\mid i\leq k\rangle$.
\end{definition}
One can check that if $v,w\in A_i$ then $p\cat \langle v,w\rangle=p\cat \langle w, v\rangle$. This permits to show that $\leq$  is a transitive partial order relation on $\mathbb{R}_u$.

\begin{remark}
 We point out that the map $\pi$ representing $\kappa$  can be used to establish a projection between $\mathbb{R}_u$ and the usual Radin forcing.
\end{remark}

\begin{lemma}[Some properties of $\mathbb{R}_u$]\label{Factoring}\hfill
    \begin{enumerate}
        \item $\mathbb{R}_u$ is a $\kappa$-centered poset;
        \item for each $p\in \mathbb{R}_u$ and $i<\ell(p)$, 
$$\mathbb{R}_{u}/p\simeq(\mathbb{R}_{u_i}/p\restriction i+1)\times(\mathbb{R}_{u}/p\restriction [i+1,\ell(p));$$
\item for each $p\in \mathbb{R}_u$ the poset $\langle \mathbb{R}_u/p,\leq^*\rangle$ is $\pi(\sigma_{u^p_0})$-directed-closed.\qed
    \end{enumerate}
\end{lemma}

Next  we verify the \emph{Strong Prikry property} for  $\mathbb{R}_u.$ For this we need the notion of a \emph{fat tree} which, to our understanding, is due to Gitik \cite[\S5]{Gitik-handbook}. 
\begin{definition}[Fat trees]
 Let $n<\omega$ and $w\in \mathcal{MS}$.   A tree $T\subseteq [\mathcal{MS}\cap V_{\kappa_w}]^{\leq n}$ consisting of $\prec$-increasing sequences is called  \emph{$w$-fat} if it is either the empty tree $\varnothing$ or for each $\langle v_0,\dots, v_k\rangle\in T$ with $k<n$, $$\text{ $\Succ_{T}(\langle v_0,\dots, v_k\rangle)\in w(\alpha)$ for some $\alpha< \ell(w)$.}$$
 

  Given a $w$-fat tree $T$ we denote by $\mathcal{B}(T)$ the maximal branches of $T$.
\end{definition}

\begin{lemma}[Strong Prikry property]
Let $p\in \mathbb{R}_u$ and $D\s \mathbb{R}_u$ be a dense open set. There is $p\leq^* p^*$, $I\s \ell(p)$ and ${\mathcal{T}}=\langle {T}_i\mid i\in I\rangle$ such that:
\begin{enumerate}
    \item ${T}_i$ is a $u^p_i$-fat tree.
   \item For each $\langle \vec v_i\mid i\in I\rangle$ with $\vec v_i\in \mathcal{B}({T}_i)$,
    $p^*\cat \langle \vec v_0,\dots,\vec{v}_{\max(I)}\rangle\in D.$
\end{enumerate}
\end{lemma}
\begin{proof}
Let $p$ and $D$ be as in the statement of the lemma. To streamline the argument let us assume that $p=\l \langle u,A\rangle\r$. The general argument follows  combining this base case with the factoring lemma (see Lemma~\ref{Factoring}).

For  $q\in\mathbb{R}_u$ and $n\leq \ell(q)$ we denote by $L_n(q)$ the \emph{$n$th-tail of $q$}; namely, $$L_n(q):=q\restriction\ell(q)-n.$$

In the first part of the proof we define by induction a $\leq^*$-increasing sequence $\langle p^n\mid n<\omega\rangle$ of conditions with $p^0:=p$ such that for each $n\geq 1$ and each condition $p^n\leq q\in D$ with $\ell(q)\geq n$ the following hold:
\begin{enumerate}
	\item There is  a $u$-fat tree $T^q$  of height $n$ with $L_n(q)\prec w$ for all $\langle w\rangle\in T^q$;
	\item For all $\vec{w}\in\mathcal{B}(T^q)$ there is $q_{\vec{w}}\in D$ such that $L_n(q)^\smallfrown(p^n\cat\vec{w})\leq^*q_{\vec{w}}$.
\end{enumerate}
Clearly the above holds for any $q\in D$ such that $p^0\leq q$ as witnessed by the tree $T^q:=\varnothing$. Suppose by induction that $p^{n}=\langle (u, A^{p^n})\rangle$ has been defined. 

\begin{claim}
	There is $p^{n}\leq^* p^{n+1}$ such that for each condition $p^{n+1}\leq q\in D$ with $\ell(q)\geq n+1$  Clauses~(1) and (2) above hold. 
\end{claim}
\begin{proof}[Proof of claim]
	Denote $\mathcal{L}_{n+1}(p^n):=\{L_{n+1}(q)\in V_\kappa\mid q\leq p^n,\, \ell(q)\geq n+1\}$. 
	
	For a fixed $L\in\mathcal{L}_{n+1}(p^n)$ denote by $A_0(L)$ the collection of all $v\in A^{p^n}$ for which there is a $u$-fat tree $T^{v}$ of height $n$ (with $v\prec w$ for all $\langle w\rangle\in T^{v}$) such that for each $\vec{w}\in \mathcal{B}(T^v)$, there is  $q_{\vec{w},v}\in D$ with $L^\smallfrown (p^n\cat \langle v,\vec{w}\rangle)\leq^* q_{\vec{w},v}$. 
	
	Denote $A_1(L):=A^{p^n}\setminus A_0(L)$. For each $\alpha<\ell(u)$ let $i_{\alpha,L}<2$ be the unique index witnessing $A_{i_{\alpha,L}}(L)\in u(\alpha)$. Define $A^{n+1}_\alpha:=\diagonal_{L\in\mathcal{L}_{n+1}(p^n)} A_{i_\alpha, L}(L)$ and
	$$ \textstyle  A^{p^{n+1}}:=(\bigcup_{\alpha<\ell(u)}A_\alpha^{n+1})\cap A^{p^n}.$$
	We claim that $p^{n+1}:=\langle \langle u,A^{p^{n+1}}\rangle\rangle$ is the sought condition. To show this let $p^{n+1}\leq q\in D$ be with $\ell(q)\geq n+1$. We shall find a $u$-fat tree $T^q$ with height $n+1$ witnessing Clauses~(1) and (2) with respect to $p^{n+1}$.

	Since $p^n\leq q$  we can use our induction hypothesis to find a $u$-fat tree $T^q$ of height $n$ such that 
	for all $\vec{w}\in \mathcal{B}(T^q)$ there is $q_{\vec{w}}\in D$ with $$L_n(q)^\smallfrown (p^n\cat \vec{w})\leq^* q_{\vec{w}}.$$ In turn, $L_n(q)$ decomposes as $L_{n+1}(q)^\smallfrown \langle v, B\rangle$ for some $v\in A^{p^{n+1}}$. Thus, 
	 $$L_{n+1}(q)^\smallfrown (p^n\cat \langle v, \vec{w}\rangle)\leq^* q_{\vec{w}}$$ and $v\in A^{n+1}_\alpha$ for some $\alpha<\ell(u)$. This means that $$v\in A^{n+1}_\alpha\cap A_{0}(L_{n+1}(q))$$ and as a result $i_{\alpha, L_{n+1}(q)}=0$. Thus, by definition, for each $v\in A^{n+1}_\alpha$ with $L_{n+1}(q)\prec v$ there is a $u$-fat tree $T^{v}$ of height $n$ (with $v\prec w$ for all $\langle w\rangle\in T^v$) such that for each $\vec{w}\in \mathcal{B}(T^v)$ there is $q_{\vec{w},v}$ with
	 $$L_{n+1}(q)^\smallfrown (p^n\cat\langle v, \vec{w}\rangle)\leq^* q_{\vec{w}, v}\in D.$$ 
	Clearly $q_{\vec{w},v}$ is $\leq^*$-compatible with $L_{n+1}(q)^\smallfrown (p^{n+1}\cat\langle v, \vec{w}\rangle)$. Thus it is harmless to assume that $q_{\vec{w},v}$ is in fact $\leq^*$-stronger than this latter condition. Thus it suffices to take  $T^q:=\{\langle v\rangle^\smallfrown\vec{w}\mid v\in A^{n+1}_\alpha,\, L_{n+1}(q)\prec v,\,\vec{w}\in T^v\}$.
\end{proof}
The above procedure defines a $\leq^*$-decreasing sequence $\langle p^n\mid n<\omega\rangle$ and we can let $p^\omega$ a $\leq^*$-lower bound for it. This condition  allows us to get rid of the dependence on the lower parts. More formally:  If $p^\omega\leq q\in D$ is a condition (say with  $\ell(q)=n$) we can use the defining property of $p^n$ to find a $u$-fat tree $T^q$ of height $n$ such that for all $\vec{w}\in\mathcal{B}(T^q)$ there is $q_{\vec{w}}\in D$ such that  $p^n\cat \vec{w}\leq^* q_{\vec{w}}$ (here we have used that $L_{n}(q):=\emptyset$). Once again, since $q_{\vec{w}}$ and $p^\omega\cat\vec{w}$ are $\leq^*$-compatible we may assume that $p^\omega\cat\vec{w}\leq^* q_{\vec{w}}$.

\smallskip


Fix a condition $q\in D$  with $p^\omega\leq q$.
\begin{claim}
	There is $p^\omega\leq^* p^{*}$  and a $u$-fat tree $S$ of length $\ell(q)$ such that  $p^{*}\cat \vec{w}\in D$ for all $\vec{w}\in S$.
\end{claim}
\begin{proof}[Proof of claim]
Let $p^\omega\leq q\in D$ and $T$ be a $u$-fat tree of height $\ell+1$ for which there is $q_{\vec{w}}\in D$  with $p^\omega\cat\vec{w}\leq^* q_{\vec{w}}$. Note that $q_{\vec{w}}$ takes the form
$$\langle (w_0,B_0(\vec{w})),\dots, (w_{\ell}, B_\ell(\vec{w})),(u,B(\vec{w}))\rangle.$$
For each $j\leq \ell+1$ let us denote $$T\restriction j:=\{\vec{v}\in[\mathcal{MS}]^{<\omega}\mid \vec{v}=\vec{w}\restriction j\; \text{for some}\, \vec{w}\in T\}$$
and, for each $\vec{v}\in T\restriction j$, denote 
$$T_{\vec{v}}:=\{\vec{\mu}\in [\mathcal{MS}]^{<\omega}\mid \vec{v}{}\,^\smallfrown\vec{\mu}\in T\}.$$

Fix $i\leq \ell$. For each $\vec{z}\in T\restriction (i+1)$ we shall be interested in the map
$$B_i(\vec{z}{}\,^\smallfrown \langle\cdot\rangle)\colon T_{\vec{z}}\rightarrow \mathcal{F}(z_i)$$ given by $\vec\mu\mapsto B_i(\vec{z}{}\,^\smallfrown \vec{\mu})$. Since all the measures involved in $T_{\vec{z}}$ are $\kappa$-complete we can find a $u$-fat tree $S(\vec{z})\s T_{\vec{z}}$ of height $(\ell+1)-i$ such that when restricting the above map to it becomes constant with value $B_i(\vec{z})$.

Let $S_i$ denote the $u$-fat tree of height $\ell$ such that
\begin{itemize}
    \item $(S_i)\restriction(i+1)= T\restriction (i+1)$;
    \item $(S_i)_{\vec{z}}=S(\vec{z})$ for all  $\vec{z}\in T\restriction i+1$.
\end{itemize}
\smallskip

For each $\vec{v}\in T\restriction i$  there is $\alpha(\vec{v})<\ell(u)$ such that $\mathrm{Succ}_{T}(\vec{v})\in u(\alpha(\vec{v}))$ 
thus the set $B_i(\vec{v})_{<\alpha(\vec{v})}:=j(z\mapsto B_i(\vec{v}\,{}^\smallfrown \langle z\rangle))(u\restriction \alpha(\vec{v}))$ belongs to $\mathcal{F}(u\restriction \alpha(\vec{v}))$.

Similarly, we define the $u(\alpha(\vec{v}))$-large set $$B_i(\vec{v})_{=\alpha(\vec{v})}:=\{z\in \mathrm{Succ}_{T}(\vec{v})\mid B_i(\vec{v}\,{}^\smallfrown \langle z\rangle)= B_i(\vec{v})_{<\alpha(\vec{v})}\cap V_{\kappa_z} \}.$$
Finally, let $B_i(\vec{v})_{>\alpha(\vec{v})}:=\{z\in A^{p^\omega}\mid \ell(z)>\alpha(\vec{v})\}$ and $$B_i(\vec{v}):=B_i(\vec{v})_{<\alpha(\vec{v})}\cup B_i(\vec{v})_{=\alpha(\vec{v})}\cup B_i(\vec{v})_{>\alpha(\vec{v})}.$$ To amalgamate all of these $B_i(\vec{v})$ we take diagonal intersections; namely, 
$$B_i:=\{z\in \mathcal{MS}\mid \forall \vec{v}\in T\restriction i\, (\vec{v}\prec z\,\Rightarrow\, z\in B_i(\vec{v}))\}.$$
It is routine to check that $B_i\in \mathcal{F}(u)$. 

\smallskip

In the end, we let $B:=\bigcap_{i\leq \ell} B_i$ and $S:=(\bigcap_{i\leq \ell} S_i)\cap B$. We claim that   $p^*:=\langle (u,B)\rangle$ together with $S$ satisfy the statement of the claim. For this it suffices to show that if  $\vec{w}\in \mathcal{B}(S)$ then $q_{\vec{w}}\leq^*p^*\cat\vec{w}$. 

\smallskip

For each $\vec{w}\in \mathcal{B}(S)$ we have 
$$p^*\cat\vec{w}:=\langle (w_0, B^*_0),\dots, (w_n, B^*_\ell), (u, B^*_{\ell+1})\rangle$$
where $B^*_{i}:=\{v\in B\cap V_{\kappa_{w_{i}}}\mid w_{i-1}\prec v\,\wedge\, \ell(v)<\ell(w_{i})\}$ for $i\leq \ell+1$.\footnote{Here we agree that $w_{-1}=w_{\ell+1}$ are the empty sequence. }

Let us check that $B^*_{i}\s B_i(\vec{w})$ for $i\leq \ell$ -- the argument showing $B^*_{\ell+1}\s B(\vec{w})$ is similar. First,  $\langle w_{i+1},\dots, w_{\ell}\rangle\in (S_i)_{ \vec{w}\restriction i+1}$ so  $$B_i(\vec{w})=B_i(\vec{w}\restriction i+1).$$
Second, $w_i\in \mathrm{Succ}_{S_i}(\vec{w}\restriction i)=\mathrm{Succ}_{T}(\vec{w}\restriction i)\in u(\alpha(\vec{w}\restriction i))$. In particular, $\ell(w_i)=\alpha(\vec{w}\restriction i).$ Now let $v\in B^*_i$. By definition of diagonal intersection, $$v\in B\cap V_{\kappa_{w_i}}\s B_i(\vec{w}\restriction i)\cap V_{\kappa_{w_i}}.$$ Also, $v$ has length ${<}\alpha(\vec{w}\restriction i)$ so it belongs to $$B_i(\vec{w}\restriction i)_{<\alpha(\vec{w}\restriction i)}\cap V_{\kappa_{w_i}}= B_i(\vec{w}\restriction i+1)=B_i(\vec{w}).$$
For the first of these equalities we used that $w_i\in B_i(\vec{w}\restriction i)_{=\alpha(\vec{w}\restriction i)}$.

Thereby we have showed that $B^*_i\s B_i(\vec{w})$ as sought.
\end{proof}
The above claim completes the verification of the lemma.
\end{proof}

Let us now describe the main combinatorial object introduced by $\mathbb{R}_u$:
\begin{definition}
    Let $G\subseteq \mathbb{R}_{u}$ be a $V$-generic filter. Denote
    \begin{itemize}
        \item $\mathcal{MS}_G:=\{v\in \mathcal{MS}\mid \exists p\in G \  \exists i< \ell(p)\, \ v={u}^p_i\}$;
        \item $\Sigma_G:=\{\sigma_{v}\mid v\in \mathcal{MS}_G\}$;
        \item $C_G:=\{\kappa_v\mid v\in \mathcal{MS}_G\}$.
    \end{itemize}
\end{definition}
\begin{proposition}
    There is $\xi<\kappa$ such that $\Sigma_G\setminus \xi$ is a totally discontinuous sequence; namely, $\sup(\Sigma_G\cap\alpha)<\alpha$ for all  $\alpha\in \Sigma_G\setminus \xi$. 
    
    Moreover, for each such $\alpha$, $\Sigma_G\cap \alpha\subseteq \pi(\alpha)<\alpha$.\footnote{Recall that $\pi\colon \kappa\rightarrow\kappa$ is the function representing $\kappa_u$ via $\sigma_u$ (see p.\pageref{pi}).}
\end{proposition}
\begin{proof}
Let us go for the moreover assertion. Recall that $j\colon V\rightarrow M$ is a constructing embedding for $u$ and that $j(\pi)(\sigma)=\kappa<\sigma$.  It follows that $X=\{v\in \mathcal{MS}\mid \pi(\sigma_v)<\sigma_v\}$ belongs to $\mathcal{F}(u)$. In particular, the set of conditions $p$ with $A^p_{\ell(p)}\s X$ is $\leq^*$-dense. Let $p\in G$ be a condition with that property and define $\xi:=\sigma_{u^p_{\ell(p)-1}}$. For each $\alpha\in \Sigma_G\setminus (\xi+1)$ there is $v\in \mathcal{MS}_G$ such that $\alpha=\sigma_v$ (note that $v$ must come from $X$). Let $q\in G$ witnessing this. For each $\beta\in (\Sigma_G\cap \alpha)\setminus (\xi+1)$ we can let $q\leq r$ in $G$ such that $\sigma_w=\beta$. Notice that because $\beta<\alpha$ it must be that $w$ is mentioned in $r$ before $v$ (and, once again, $w\in X$). By definition of the poset this implies that $w\prec v$, which in turn yields $\beta=\sigma_w<\pi(\sigma_v)=\pi(\alpha)$, as needed.
\end{proof}
\begin{remark}
    On the contrary, standard arguments show that $C_G$ is a club on $\kappa$. This is the Radin club introduced by the normal Radin  induced by $\pi$.
\end{remark}

Arguing similarly one can prove the next propositions:
\begin{proposition}\label{Prop: special for limits}
    Suppose that $\alpha<\ell(u)\leq\kappa$ and let $A\in\mathcal{F}(u)$. Then there is $\xi<\kappa$ such that $(\mathcal{MS}_G\cap \{v\in \mathcal{MS} \mid \ell(v)\geq \alpha\})\setminus V_\xi\subseteq A$.
\end{proposition}
\begin{proposition}\label{prop: limit index}
    Suppose that $v\in\mathcal{MS}_G$ is such that $\kappa_v$ has limit index in the natural enumeration of $C_G$ then $\ell(v)>1$. 
\end{proposition}
\begin{corollary}\label{cor: special for limits}
Suppose that $2<\ell(u)\leq \kappa$ and let $A\in\mathcal{F}(u)$. Then there is $\xi<\kappa$ such that $\{\kappa_v\in C_G\mid \text{$v$ has limit index in $C_G$}\}\setminus \xi\subseteq \{\kappa_v\mid v\in A\}$.
\end{corollary}

\section{Adding Cohen functions to every limit point}\label{sec: AddingCohends}

Suppose that $\kappa$ is a $\mathcal{P}_2\kappa$-hypermeasurable cardinal. In this section we employ our Radin forcing from \S\ref{SectionNonNormalRadin} to shoot a club $C\s \kappa$ with $\text{otp}(C)=\omega_1$ whose limit points $\alpha$ carry a  Cohen generic set $c_\alpha\s \alpha$. Our main result here is Theorem~\ref{ManyCohens}. The next prelimminary result paves the way to Theorem~\ref{ManyCohens}.
\begin{lemma}\label{Preparation}
    Assume the $\mathrm{GCH}$ holds and that $\kappa$ is $\mathcal{P}_2\kappa$-hypermeasurable cardinal. There is a cofinality-preserving generic extension $V[G]$ where:
    \begin{enumerate}
        \item $\mathrm{GCH}$ holds;
        \item There is a $\mathcal{P}_2\kappa$-hypermeasurable embedding $j:V[G]\rightarrow M[H]$ and an ordinal  $\sigma\in (\kappa,j(\kappa))$ such that the (non-normal) measure $$W:=\{X\in \mathcal{P}(\kappa)^{V[G]}\mid \sigma\in j(X)\}$$ witnesses that its  Tree Prikry forcing $\mathbb{T}_W$ projects onto $\Add(\kappa,1)$.
    \end{enumerate}
\end{lemma}
\begin{proof}
   Let us begin by fixing an elementary embedding $j\colon V\rightarrow M$ arising from a $(\kappa,\kappa^{++})$-extender $E$ -- this is possible in that $\kappa$ is $\mathcal{P}_2\kappa$-hypermeasurable. Let us denote by $i\colon V\rightarrow N$ the ultrapower by the normal measure on $\kappa$ inferred from $j$. As usual, this yields a factor embedding $k\colon N\rightarrow M$ defined by $k(i(f)(\kappa)):=j(f)(\kappa).$ Standard arguments involving the GCH show that $k$ has width ${\leq}\kappa^{++}_N$ and that $\crit(k)=\kappa^{++}_N$.

    \smallskip

     We go for the forcing preparation spelled out in \cite[\S~7]{TomMaster}. Namely, our forcing extension will be given by the Easton-supported iteration 
     $$\langle \mathbb{P}_\alpha, \dot{\mathbb{Q}}_\beta\mid \alpha\leq \kappa,\, \beta<\kappa\rangle$$
     defined as follows: For each $\alpha<\kappa$, $\dot{\mathbb{Q}}_\alpha$ is a $\mathbb{P}_\alpha$-name for the trivial forcing unless $\alpha$ is inaccessible, in which case it is a $\mathbb{P}_\alpha$-name for the lottery sum of $$\{\Add(\alpha,1),\{\one\}\}.$$

     Let $G\s \mathbb{P}_\kappa$ a $V$-generic filter. It is easy to lift the embeddings $i\colon V\rightarrow N$ and $k\colon N\rightarrow M$ inside $V[G]$.  Indeed, $i$ lifts to $i\colon V[G]\rightarrow N[G\ast\{\one\}\ast H]$ where $H\in V[G]$ is a $N[G]$-generic filter for the tail poset $i(\mathbb{P}_\kappa)/G\ast \{\one\}.$\footnote{The construction of a $H$ in $V[G]$ is standard employing the GCH and the high degree of closure of the tail forcing.} Since this tail forcing is more close than the width of the embedding $k$ one can lift this latter to $k\colon N[G\ast\{\one\}\ast H]\rightarrow M[G\ast \{0\}\ast k``H]$. Incidentally,
     $$j\colon V[G]\rightarrow M[G\ast \{0\}\ast k``H].$$
     For reasons that will become clear shortly we have to prepare a $M[j(G)]$-generic filter $g_{j(\kappa)}$ for $\Add(j(\kappa),1)_{M[j(G)]}$. This is done as before: First, one can cook up a $N[i(G)]$-generic $g_{i(\kappa)}\in V[G]$ for  $\Add(i(\kappa),1)_{N[j(G)]}$ (for this one employs the GCH). Second, we can trasnfer $g_{i(\kappa)}$ to $g_{j(\kappa)}$ through $k$; clearly, $g_{j(\kappa)}\in V[G]$. In addition, we can alter $g_{j(\kappa)}$ so that $g_{j(\kappa)}(0)=\kappa.$

     \smallskip

     Now we go to the second ultrapower of our initial extender $E$; specifically, let us consider $j_{1,2}\colon M\rightarrow M_2\simeq \mathrm{Ult}(M,j(E))$. Eventually, we would like to lift $j_2:=j_{1,2}\circ j$ and for this it would suffice to lift $j_{1,2}$ under $j(\mathbb{P}_\kappa)$ (as the other embedding has been already lifted).  To this end, note that
     $$j_2(\mathbb{P}_\kappa)\downarrow p \simeq j(\mathbb{P}_\kappa)\ast \dot{\Add}(j(\kappa),1)\ast \dot{\mathbb{T}}_{(j(\kappa), j_{2}(\kappa))}$$
     where $p$ is the condition opting for Cohen forcing at stage $j(\kappa)$.

     Clearly, $\Add(j(\kappa),1)_{M_2[j(G)]}=\Add(j(\kappa),1)_{M_1[j(G)]}$ so $j_{1,2}$ lifts to 
     $$j_{1,2}\colon M[j(G)]\rightarrow M_2[j(G)\ast g_{j(\kappa)}\ast T]$$
     for some generic  $T$ for the tail forcing. As before, we can construct this $T$ inside $M[j(G)]$ by factoring through the normal ultrapower of $j_{1,2}$ and using the GCH in the model $M[j(G)].$ Therefore, the above lives inside $V[G]$.

\smallskip

This produces an elementary embedding $j_2\colon V[G]\rightarrow M_2[j_2(G)]$ such that:

\begin{claim}\label{ClaimPreparingManyCohens}
    The following hold for $j_2$ in $V[G]$:
    \begin{enumerate}
        \item $j_2$ is a $\mathcal{P}_2\kappa$-hypermeasurable embedding;
        \item Let $W:=\{X\in \mathcal{P}(\kappa)^{V[G]}\mid j(\kappa)\in j_2(X)\}$. Then, $\mathrm{Cub}_\kappa\s W$ and  $$\mathcal{C}:=\{\alpha<\kappa\mid \exists f_\alpha\, (f_\alpha\,\text{is Cohen generic over $V[G_\alpha]$})\}\in W;$$
    
        \item $j_2(\pi)(j(\kappa))=\kappa$ where $\pi\colon \kappa\rightarrow \kappa$ is the  function defined as
        $$
        \pi(\alpha):=\begin{cases}
           f_\alpha(0), & \text{if $\alpha\in\mathcal{C}$;}\\
           0, & \text{otherwise.}
        \end{cases}
        $$
    \end{enumerate}
\end{claim}
 \begin{proof}[Proof of claim]

        (1) Note that $V_{\kappa+2}\s M_2$ because $V_{\kappa+2}\s M_1$ and both $M_1$ and $M_2$ agree up to $(V_{j(\kappa)})^{M_1}$. Also, $\mathbb{P}_\kappa$ is $\kappa$-cc so that $V[G]_{\kappa+2}=V_{\kappa+2}[G]$ From this we can easily infer that $V[G]_{\kappa+2}\s M_2[j_2(G)]$.

        (2) Let $C\in (\mathrm{Cub}_\kappa)^{V[G]}$. By $\kappa$-ccness of $\mathbb{P}_\kappa$ there is $D\s C$ in $(\mathrm{Cub}_\kappa)^V$. By normality, $j(D)$ belongs to the normal measure on $j(\kappa)$ inferred (in $V$) from $j_{1,2}$. From this it follows right away that $C\in W$.

        The claim about $\mathcal{C}$ is evident because $g_{j(\kappa)}$ was chosen to be a Cohen generic over $M_{2}[j(G)]=M_{2}[j_2(G)_{j(\kappa)}]$.

        (3) This follows from our choice that $g_{j(\kappa)}(0)=\kappa.$
    \end{proof}
To complete proof of Lemma~\ref{Preparation} it remains to show that (in $V[G]$) the Tree Prikry forcing $\mathbb{T}_W$ corresponding to $W$ projects onto $\Add(\kappa,1)_{V[G]}$. 
  

 \begin{claim}\label{GluingTheCohens-PrikryCase}
     There is a projection between $\mathbb{T}_{W}$ and $\Add(\kappa,1)_{V[G]}$.
 \end{claim}
 \begin{proof}
 
 Let $\langle \kappa_n\mid n<\omega\rangle$ be a Prikry sequence over $V[G]$.  We shall show that this induces a $V[G]$-generic for the Cohen poset. Let $A\in V[G]$ be a maximal antichain for $\Add(\kappa,1)_{V[G]}$. For each $\alpha<\kappa$ regular and $p\in \Add(\alpha,1)_{V[G]}$ let $\beta(p)<\alpha$ be the first ordinal such that $p$ is compatible with a member of $A\cap \Add(\beta(p),1)_{V[G]}$. Define a function $f\colon \kappa\rightarrow \kappa$ by
 $$\textstyle f(\alpha):=\sup_{p\in \Add(\alpha,1)_{V[G]}} \beta(p)$$
 whenever $\alpha$ is a regular cardinal; declare it to be $0$ otherwise.

 Let $C(f)$ be the closure points of $f$. This set is a club in $V[G]$ so $C\cap \mathcal{C}\in W$ (here $\mathcal{C}$ is as in the previous claim). Let $1\leq n_*<\omega$ be such that $\kappa_n\in C\cap \mathcal{C}$ for all $n\geq n_*$ and define (in $V[G]$) 
 $$\textstyle f^*:=\bigcup_{n\geq n_*} f_{\delta_n}\restriction [\delta_{n-1},\delta_n).$$

 One can show that $f^*$ is $\Add(\kappa,1)_{V[G]}$-generic. We refer the reader to  \cite[Proposition 7.3]{TomMaster} for details. 
 \end{proof}
The above claim completes the proof of the lemma.
\end{proof}

 Let $V^*$ denote the model obtained in Lemma~\ref{Preparation} and   $j^*\colon V^*\rightarrow M^*$ the corresponding $\mathcal{P}_2\kappa$-hypermeasurable embedding.  From now on $V^*$ will be  our ground model.  Using $j^*$, $\pi$ and $\sigma$ from Lemma~\ref{Preparation} we derive the corresponding measure sequence $u$; namely, $u(0):=\l \sigma\r$ and for each $\xi\geq 1$,  
 $$u(\xi):=\{X\s V_\kappa\mid u\restriction \xi\in j^*(X)\}.$$
Note that $u(1)$ is essentially $W$; more precisely, $$\text{$X\in W$ if and only if $\{\langle \alpha\rangle\mid \alpha\in X\}\in u(1)$.}$$
\begin{theorem}\label{ManyCohens}
    Let $G^*\subseteq\mathbb{R}_u$ be $V^*$-generic. For all except bounded-many $\alpha\in \lim(C_{G^*})\cup\{\kappa\}$ with $\mathrm{cf}(\alpha)^{{V^*[G^*]}}=\omega$ there is a $V^*$-generic Cohen function $f_\alpha\in V^*[G^*]$ for $\Add(\alpha,1)_{V^*}$.
\end{theorem}
\begin{proof}
    Let us begin noting that\footnote{In a slight abuse of notation, here we have identified $v(1)$ with the corresponding measure on $\kappa_v$ rather than on $V_{\kappa_v}$.} $$X:=\{v\in\mathcal{MS}\mid \exists\pi\colon\mathbb{T}_{v(1)}\rightarrow \Add(\kappa_v,1)\,\text{projection in }V^*\}\in\mathcal{F}(u).$$
    Indeed, this is because $\mathbb{P}(W)$ projects onto $\Add(\kappa,1)$ and this is correctly computed by the model $M^*$. By Proposition~\ref{Prop: special for limits} there is $\beta<\kappa$ such that 
    $$(\mathcal{MS}_{G^*}\cap \{v\in\mathcal{MS}\mid \ell(v)>1\})\setminus V_\beta\s X.$$
    Let $\alpha\in \lim(C_{G^*})\cup\{\kappa\}$ be with $\alpha>\beta$ and $\cf(\alpha)^{V[G^*]}=\omega$. By definition there is $v\in \mathcal{MS}_G$ such that $\alpha=\kappa_v$ and, clearly, $\alpha$ must have limit index in the enumeration of $C_{G^*}$. By Proposition~\ref{prop: limit index}, $v$ is a measure sequence with 
    $\ell(v)>1$ so the above inclusion gives $v\in X$. Thus, $\mathbb{T}_{v(1)}$ projects onto $\Add(\kappa_v,1)$. Next we show that a bounded piece of the Radin club $\langle \kappa_\alpha\mid \alpha<\omega^{\ell(u)}\rangle$ can be used to produce a generic for $\Add(\kappa_v,1)$.

    \smallskip

    Let $\langle v_n\mid n<\omega\rangle\s \mathcal{MS}_{G^*}$ of length $1$ such that $\sup_{n<\omega}\kappa_{v_n}=\alpha.$

    \begin{claim}
        $\l\alpha_n\mid n<\omega\r$ is a $\mathbb{T}_{v(1)}$-generic sequence.
    \end{claim}
    \begin{proof}[Proof of claim]
        Let us use the Mathias criterion for the Tree-Prikry forcing from \cite{TomTreePrikry}. 
        Let $A\in v(1)$ and  $p\in G^*$ be condition mentioning $v$; say at coordinate $i$. 
        Shrink $A^p_i$ to $A^*_i\subseteq A$, and extend $p\leq^*p^*$ so that the $i$-th set in $p^*$ is $A^*_i$. Note that $p^*$ forces that every successor element element of $C_{G^*}$ in the interval $(\sigma_{\bar{u}^p_{i-1}},\alpha)$ is in $A$. By density, we can find such a condition in $G^*$ and so a tail of the $\alpha_n$'s falls in $A$.
    \end{proof}
    Since $\mathbb{T}_{v(1)}$ projects on $\Add(\alpha,1)$,  there is a generic Cohen  function $f\in V^*[\l\alpha_n\mid n<\omega\r]$ and since $V^*[\l\alpha_n\mid n<\omega\r]\subseteq V^*[G^*]$ we are done.
\end{proof}

\begin{corollary}\label{ManyCohenswithRadin}
    If $\ell(u)=\omega_1$ then below a certain condition $p\in \mathbb{R}_{u}$ the poset $\mathbb{R}_{u}/p$ adds a  $V^*$-generic Cohen function to every limit point of the generic club $C_G$.
\end{corollary}
After an appropriate preparation, we have just shown that forcing a $\mathbb{R}_u$-generic club $C\s \kappa$ of order-type $\omega_1$ automatically adds a $V$-generic Cohen subset to every limit point of $C$. Clearly, all those points have countable cofinality in $V[C]$. However, do we add a Cohen subset to $\kappa$? Or, alternatively, if we employ $\mathbb{R}_u$ to add a generic club $C\s \kappa$ of order-type $\omega_2$, do the limit points of $C$ of cofinality $\omega_1$ carry a $V$-generic Cohen subset in $V[C]$? Suppose that $\vec{\kappa}=\langle \kappa_\alpha\mid \alpha<\omega_1\rangle$ is a Magidor/Radin generic. If one attempts to amalgamate a Cohen generic $f$ using $\vec{\kappa}$ --similarly to what we did in Claim~\ref{GluingTheCohens-PrikryCase}-- this will not work: On one hand,  the restriction $f\restriction\kappa_\omega$ is generic over the ground model by virtue of Claim~\ref{GluingTheCohens-PrikryCase}. On the other hand, if $f$ is a $V$-generic for $\mathrm{Add}(\kappa,1)$ then $f\restriction\alpha\in V$ for all $\alpha<\kappa$.  This restriction is in fact hiding the impossibility for a Radin-like forcing to introduce a fresh subset of $\kappa$, provided this latter cardinal changes its cofinality to ${\geq}\omega_1$. 

\begin{definition}
Let $\mathbb{P}$ be a forcing poset and $G\s \mathbb{P}$ a $V$-generic filter. 
   A set $x\s \kappa$ is called \emph{$(V,V[G])$-fresh}  if $x\in V[G]$ and $x\cap \alpha\in V$ for all $\alpha<\kappa$. 
\end{definition}

The next fact appears in \cite[Proposition~1.3]{BenNeria} where the author gives credit to Cummings and Woodin. The proof in the non-normal scenario is verbatim the same as the one provided by Ben-Neria -- one simply replaces the usual diagonal intersection by our revised definition employing the order $\prec$ of Definition~\ref{OrderBetweenMeasures}:
\begin{fact}[Cummings and Woodin]
  Assume  $u\in\mathcal{MS}$ has $\cf(\ell(u))\geq \omega_1$. Then, the trivial condition of $\mathbb{R}_{u}$ forces that $\text{$``\forall \tau\s \kappa\, (\text{$\tau$ is fresh\, $\Rightarrow$\,$\tau\in \check{V})$}$''.}$
\end{fact}
The same  holds true for the non-normal Magidor forcing defined in  \S\ref{Sec: non-normalMagidor}.
\begin{corollary}\label{CorollaryFresh}
    If $x$ is $(V,V[G])$-fresh then $\cf^{V[G]}(\sup(x))=\omega$.
\end{corollary}
\begin{proof}
The proof is by induction on $\sup(x)$. Denote by $\lambda=\cf^V(\sup(x))$ and let $\langle \delta_\alpha\mid \alpha<\lambda\rangle\in V$ be a cofinal sequence in $\sup(x)$. 

\smallskip
\underline{Case $\lambda\geq \kappa^+$}: For each $\alpha<\lambda$, since $x$ is fresh, we can let $p_\alpha=\vec{d}_\alpha{}^\smallfrown(u,A_\alpha)\in G$ deciding the value of $\dot{x}\cap \delta_\alpha$. By passing to an unbounded subset of $\lambda$ we can assume that $\vec{d}_\alpha=\vec{d}_*$. Next define  $$y=\{\nu<\kappa\mid \exists A\in\mathcal{F}(u), \ \vec{d}_*{}^{\smallfrown}(u,A)\Vdash \nu\in \dot{x}\}.$$
Then $y\in V$ and we claim that $y=x$. Indeed, if $\nu\in y$ then, for some $\alpha<\lambda$, $\nu<\delta_\alpha$ and there is $A$ such that $p'=\vec{d}_*{}^{\smallfrown}(u,A)\Vdash \nu\in \dot{x}$. Since $\vec{d}_*{}^{\smallfrown}(u,A_\alpha)\Vdash \dot{x}\cap\delta_\alpha=x\cap\delta_\alpha$ is compatible with $p'$, it must be that $\nu \in x\cap\delta_\alpha$ (otherwise, a common extension would have forced contradictory information). Conversely, if  $\nu\in x$ one finds $\delta_\alpha$ such that $\nu<\delta_\alpha$. Since $\vec{d}_*{}^\smallfrown(u,A_\alpha)\Vdash \nu\in x\cap\delta_\alpha=\dot{x}\cap\delta_\alpha$, it follows that $A_\alpha$ witness that $\nu\in y$.

\smallskip

\underline{Case $\lambda\leq\kappa$} 
Let $x=x_0$. We fix in $V$ a sequence $\l \phi_\alpha\mid \alpha<\lambda\r\in V$ such that $\phi_\alpha:\mathcal{P}(\delta_\alpha)\rightarrow 2^{\delta_\alpha}$ is a bijection. Let $\lambda_\alpha=\phi_\alpha(x\cap\delta_\alpha)$. By $\kappa^+$-c.c of $\mathbb{R}_u$, we can find $f:\lambda^*\rightarrow \mathcal{P}_{\kappa^+}(\lambda^*)\in V$ (where $\lambda^*=\sup\{\lambda_\alpha\mid \alpha<\lambda\}$)  such that $\lambda_\alpha\in f(\alpha)$. For each $\alpha$ let $i_\alpha<\kappa$ be such that $\lambda_\alpha$ is the $i_\alpha$-th element of $f(\alpha)$ in its increasing enumeration. We can define $i^*_\alpha$ recursively as follows $i^*_0=i_0$ and $i^*_{\alpha}=(\sup_{j<\alpha}i^*_j)+i_\alpha$. Note that $i^*_\alpha$ is increasing and $i_\alpha$ is definable from the sequence $i^*_\alpha$ (as the unique ordinal $\gamma$ such that $(\sup_{j<\alpha}i^*_j)+\gamma=i^*_\alpha$). Also note that since $\kappa$ is regular in $V$, and for each $\beta<\lambda\leq\kappa$, $\{i_\alpha\mid \alpha<\beta\}\in V$, $i^*_\alpha<\kappa$.
We conclude that the set $x_1=\{i^*_\alpha\mid \alpha<\lambda\}\subseteq \kappa$ is fresh. If $x_1$ is unbounded in $\kappa$, then by the previous proposition, $\omega=\cf^{V[G]}(\kappa)=\cf^{V[G]}(\lambda)=\cf^{V[G]}(\sup(x))$.
Otherwise, $x_1$ is bounded in $\kappa$ and we let $\kappa^*_1=\sup(\lim(C_G)\cap \sup(x_1))<\kappa$.  Then $x_1\in C_G\restriction \kappa^*_1$, and we may apply the induction hypothesis. 
\end{proof}
For a forcing notion $\mathbb{Q}$ let us  denote by $\mathrm{dist}(\mathbb{Q})$ the unique $\lambda$ such that $\mathbb{Q}$ is $\lambda$-distributive yet not $\lambda^+$-distributive. Equivalently, 
 $$\mathrm{dist}(\mathbb{Q})=\min\{\theta\in \mathrm{Card}\mid \exists \tau\in V^{\mathbb{Q}}\, \one\forces_{\mathbb{Q}} \text{$``\tau\s \mathrm{Ord}\,\wedge\,  |\tau|=\theta\,\wedge\, \tau\notin\check{V}$''}\}.$$
 Note that $\mathrm{dist}(\mathbb{Q})$ is a regular cardinal in $V^{\mathbb{Q}}$.
\begin{corollary}\label{cor: no fresh subsets}
    $\mathbb{R}_u$  projects only on forcings $\mathbb{Q}$ such that $\cf^{V[G]}(\mathrm{dist}(\mathbb{Q}))=\omega$ and therefore $\mathrm{dist}(\mathbb{Q})\in \{\omega\}\cup(\lim(C_G)\cap \cf(\omega)).$
\end{corollary}
\begin{proof}
Suppose $\mathrm{dist}(\mathbb{Q})=\lambda$.
\begin{claim}
    There is a fresh set of ordinals $A\in V^{\mathbb{Q}}\setminus V$ such that $$\cf^{V^{\mathbb{Q}}}(\sup(A))=\lambda.$$
\end{claim}
\begin{proof}[Proof of claim]
    Let $A\in V^{\mathbb{Q}}\setminus V$ be a set of ordinals with  $\lambda=|A|^{V^{\mathbb{Q}}}$. 
    Take $\rho\leq\sup(A)$ be the minimal ordinal such that $A\cap\rho\notin V$. If $\cf^{V^{\mathbb{Q}}}(\rho)<\lambda$  we would reach a contradiction with the fact of $\mathbb{Q}$ being $\lambda$ distributive. Hence it must be that $\cf^{V^{\mathbb{Q}}}(\rho)\geq \lambda$, in which case, $\cf^{V^{\mathbb{Q}}}(\rho)=\lambda$ since $A\cap\rho\in V^{\mathbb{Q}}$ is of size $\leq\lambda$ and must be unbounded in $\rho$ (by minimality of $\rho$). 
\end{proof}
  Let $A$ be a set as in the claim.  By Corollary~\ref{CorollaryFresh},  $\cf^{V[G]}(\sup(A))=\omega$, and as a result $\cf^{V[G]}(\lambda)=\omega$. Hence $\lambda$ is a regular cardinal which changed its cofinality in $V[G]$ to $\omega$, and thus $\lambda\in \{\omega\}\cup(\lim(C_G)\cap \cf(\omega)).$ 
\end{proof}

 Recall that by Theorem \ref{Thm:Projection of extender-based Radin} the non-normal Magidor forcing of \S\ref{Sec: non-normalMagidor} is a projection of the extender based Magidor-Radin forcing. Similarily, it is possible to show that the non-normal Radin forcing of this section is a projection of the extender-based Radin forcing from \cite{CarmiRadin}. 

Let us use the observation above regarding our forcing to conclude that also in the extender-based Radin and Magior/Radin there are no fresh subsets of $\kappa$. We will need to use the properness-like property of the extender-base Magidor radin forcing \cite[Lemma 4.13]{CarmiMagidorRadin}:
Assume $\chi$ is large enough, $N\prec H_{\chi}$ is an elementary submodel, and $P \in N$ is a forcing notion. A condition $p \in P$ is called $\l N,P\r$-generic if for each dense open subset $D \in N$ of $P$,
$$p\Vdash_P \check{D}\cap\lusim{G}\cap\check{N}\neq \emptyset$$

where $\lusim{G}$ is the name of the $P$-generic object.
\begin{corollary}\label{cor: no fresh subsets carmi}
    Suppose that $\bar{E}=\langle E_\xi\mid \xi<o(\bar{E})\rangle$ is an extender sequence with $\cf(o(\bar{E}))\geq\omega_1$ such that each $E_\xi$ is a $(\kappa,\lambda_\xi)$-extender and $\lambda_\xi<j_{E_0}(\kappa)$. Let 
  $\mathbb{P}_{\bar{E}}$ be either the extender-based Radin forcing or the extender-based Magidor/Radin forcing. Then, for every $V$-generic filter $G\subseteq \mathbb{P}_{\bar{E}}$  there are no $(V,V[G])$-fresh subsets of $\kappa$.
\end{corollary}
\begin{proof}
    Let us prove that if $A\subseteq\kappa$, $A\in V[G]$, then there is a sequence of $\alpha_i$'s such that $\alpha_i<j_{E_0}(\kappa)$, and $\l E_i(\alpha_i)\mid i<o(\bar{E})\r$ is $\triangleleft$-increasing, such that $A\in V[G^*]$, where $G^*$ is the projected generic for $\mathbb{M}[\vec{U}]$, $\vec{U}$ being the generalized cohere sequence derived from $\l E_i(\alpha_i)\mid i<o(\bar{E})\r$. Since $V[G^*]$ does not have fresh subsets of $\kappa$, $A$ cannot be $(V,V[G])$-fresh. Let $\l\lusim{a}_i\mid i<\kappa\r$ be a sequence of $\mathbb{P}_{\bar{E}}$-names for an enumeration of $A$ and let $N\prec H_{\chi}$ $|N|=\kappa$, $\l\lusim{a}_i\mid i<\kappa\r,\mathbb{P}_{\bar{E}}\in N$, $N$ is closed under $<\kappa$-sequences and $N\cap \kappa^+\in\kappa^+$. Then there is $p^*$ which is $(N,\mathbb{P}_{\bar{E}})$-generic \cite[Lemma  4.13]{CarmiMagidorRadin}. In particular, consider the dense open set
    $$D_i=\{p\in\mathbb{P}_{\bar{E}}\mid p\text{ decides } \lusim{a}_i\}$$ 
    then $D_i\in N$ since $\lusim{a}_i,\mathbb{P}_{\bar{E}}\in N$ and by elementarity. Let $Y=N\cap \mathfrak{D}$, then $Y\in P_{\kappa^+}(\mathfrak{D})$. Let us find in $G$ an $(N,\mathbb{P}_{\vec{E}})$-generic condition $ p^*\in G$.  \begin{claim}
        For each $i<o(\bar{E})$, it is possible to find a single $\alpha_i<j_{E_0}(\kappa)$ and a function $f_i$, such that $j_{E_i}(f_i)(\alpha_i)=mc_i(Y)$ and $\l E_i(\alpha_i)\mid i<o(\bar{E})\r$ is $\triangleleft$-increasing.
    \end{claim}
    \begin{proof}
        Fix $i<o(\bar{E})$. Find a bijection $\phi:\kappa\rightarrow [\kappa]^{<\omega}$ such that for every limit ordinal $\alpha$ of cofinality $|\alpha|$ $\phi\restriction \alpha:\alpha\rightarrow [\alpha]^{<\alpha}$.  We construct $\alpha_i$'s by induction. In $M_{E_i}$, represent $j_{E_i}(f_i)(\xi_1,...\xi_n)=Y$, where $\xi_1,...,\xi_n<\lambda_i<j_{E_0}(\kappa)$  $j_{E_i}(g_i)(\eta_1,...\eta_m)=\{\alpha_j\mid j<i\}$ and $j_{E_i}(h_i)(\zeta_1,...,\zeta_k)=\l E_j\mid j<i\r$. Let $\alpha_i=j_{E_i}(\phi)(\{\xi_1,...,\xi_m,\eta_1,...,\eta_m,\zeta_1,...,\zeta_k\})$. Using the fact that $o(\bar{E})$ is small, we see that $E_j(\alpha_j)\in M_{E_i(\alpha_i)}$ for all $j<i$ and that there is a function $f_i'$ such that $j_{E_i}(f_i')(\alpha_i)=mc_i(Y)$. 
    \end{proof} We have $R_i\in E_i(Y\cup\{\bar{\alpha}_i\})$, such that for each $\mu\in R_i$, $\bar{\alpha}_i\in \dom(\mu)$, $o(\mu)=i$, and $\mu\restriction (Y\cap \dom(\mu))=f_i(\mu(\bar{\alpha}_i)_0)$. So we can find $p^*\leq^* p_*\in G$. To simplify the notation let us assume that $p_*=\l \bigcup_{i<o(\bar{E})}R_i,f_0,\bar{E}\r$. 
    Let us define $q\restriction Y$ and $Y^q_i= \dom(f^{q\restriction Y}_i)\subseteq \dom(f^q_i)$ for each $1\leq i\leq l(q)+1$ for every $q\leq p_*$. The indented meaning of $Y_j^{q}$ is the collection of extender sequences indexing the \emph{$j$th-block} of a condition $q$ which extends a pure condition whose top block is indexed by $Y$. Recall that when we extend a condition we have to reflect/squeeze the extender sequences indexing each block -- this is exactly the meaning of $Y^q_j$.

    $p_*\restriction Y=p_*$ and $Y^{p^*}_1=Y$. Suppose that $q\restriction Y$ and $Y^q_i$'s were defined, let $\mu\in A^q_i$, define $(q^\smallfrown \mu)\restriction Y=q\restriction Y^\smallfrown (\mu\restriction Y^q_i)$ and $$Y^{q^\smallfrown \mu}_{j}=\begin{cases}
        Y^{q}_{j} & j<i.\\
        \mu[Y^{q}_{i}] & j=i.\\
        Y^{q}_{j-1} & j>i.
    \end{cases}.$$ 

    Since $Y^q_i\subseteq \dom(f^{q\restriction Y}_i)$, $\mu\restriction Y^q_i\in A^{q\restriction Y}_i$. If $q'\leq^* q$ we define $A^{q'\restriction Y}_i=\{\mu\restriction Y^q_i\mid \mu\in A^{q'}_i\}$ and $Y^{q}_j=Y^{q'}_j$. Note that for every $q$, $q\restriction Y$ is a condition and that the map $q\mapsto q\restriction Y$ respects both $\leq$ and $\leq^*$.
    \begin{claim}
        If $r\in N$ and $r\leq q$ then $r\leq q\restriction Y.$
    \end{claim}
    \begin{proof}
        By induction on $l(q)$. For $l(q)=1$, and $r\in N$ such that $r\leq q$ we have $r=\l\l f^r,A^r,\bar{E}\r\r$. Since $r\in N$, $f^r\in N$ and $N\cap\kappa^+\in \kappa^+$,  $\dom(f^r)\subseteq N\cap \mathfrak{D}=Y$. This suffices to infer that $r\leq q\restriction Y.$ 
        
        Let us provide details for the case $l(q)=2$ (the others are analogue by induction). In that case $$q={q_0}^\smallfrown \mu=\l \l f^q_1,A^q_1,e^q_1\r,\l f^q_2,A^q_2,\bar{E}\r\r.$$ If $r\leq q$, $r\in N$, we may assume that $l(r)=2$ for otherwise we can apply the induction hypothesis. So $$r=\l \l f^r_1,A^r_1,e^r_1\r,\l f^r_2,A^r_2,\bar{E}\r\r={r_0}^\smallfrown \mu'$$ where $\mu'=\mu\restriction \dom(f^r_2)$. Thus  $$\dom(f^r_1)=\{\mu'(\bar{\alpha})\mid \bar{\alpha}\in \dom(\mu'),\,o(\bar{\alpha})>0\}\subseteq \mu'[Y]\subseteq \mu[Y]=Y^{q}_1.$$ So $r\leq q\restriction Y$. 
    \end{proof}
    
    For each $i<\kappa$, let $p_i\in G\cap D_i\cap N$, then there is $p_i^*\in G$ which is a common extension of $p_*$ and $p_i$, and we consider $q_i=p^*_i\restriction Y$. By the claim $p_i\leq q_i$ and therefore $q_i\in D_i\cap G\restriction Y$ where 
     $G\restriction Y=\{q\restriction Y\mid q\in G/p_*\}$.
    
    Let $G^*$ be the $V$-generic induced from $G$ for $\mathbb{M}[\vec{U}]$, and $\vec{U}$ is the generalized coherent sequence induced from $$\l E_i(\bar{\alpha}_i)\mid i<o(\bar{E})\r.$$
We shall now prove that $G\restriction Y\in V[G^*]$ and then we can choose in $V[G^*]$, $p'_i\in G\restriction Y\cap D_i$ which suffices to compute $A\in V[G^*]$, as $G\restriction Y\subseteq G\cap D_i$.  
    
    For any condition $p\in\mathbb{M}[\vec{U}]$, we define $p'\in \mathbb{P}\restriction Y$ such that $l(p)=l(p')$ along with functions $g^p_{i,j}$ recursively as follows:

    If $p=\l\l \kappa,A\r\r$ we define $p'=\l\l f,B,\bar{E}\r\r$ where $\dom(f)=Y$ $f(\bar{\alpha})=\emptyset$ for every $\bar{\alpha}$ and $B=\bigcup_{i<o(\bar{E})}\{ \mu\restriction Y\cap \dom(\mu)\mid \mu\in R_i, f_i(\mu(\bar{\alpha}_i)_0)\in A\}$. Also  $g^p_{i,j}=f_j$. Suppose that $p'$ and $g^p_{i,j}$ where defined and consider $p^\smallfrown \beta$ where $\beta\in A^p_j$ with $o^{\vec{U}}(\beta)=i$. Then let $\mu_\beta=g^p_{i,j}(\beta)$ and let $(p^\smallfrown\beta)'=p^{'\smallfrown}\mu_\beta$ and $g^{p^\smallfrown\beta}_{i,j}=g^p_{i,j}\circ \mu_\beta^{-1}$ (for the relevant $j$). For direct extensions, we just shrink the measure one set.  Since $G\restriction Y$ is above $p_*$, the definition of the $R_i$ ensures we can recover all of $G\restriction Y$ from $G^*$. For more details see the argument of \cite[Thm. 4.2]{onCohenandPrikry}.
\end{proof}

\section{Gitik's forcing project onto Cohen forcing}\label{sec: Gitik}
In the previous section we demonstrated that the natural generalizations of Magidor/Radin forcing to the non-normal context do not introduce fresh subsets to a measurable cardinal $\kappa$ provided this latter changes its cofinality to $\omega_1$ in the corresponding generic extension. As a result none of these posets project onto any $\kappa$-distributive -- including among them Cohen forcing $\Add(\kappa,1)$. This raises an obvious question: Suppose that $\mathbb{P}$ is a cardinal-preserving forcing changing the cofinality of a measurable $\kappa$ to $\omega_1$. Is it feasible at all for $\mathbb{P}$ to project onto $\Add(\kappa,1)$? In this section we show that (once again, after a suitable preparation) the natural non-normal version of Gitik's forcing from \cite{GitikNonStationary} does project onto $\Add(\kappa,1)$. We begin with a warm-up section \S\ref{SubsectionAddingOverOmega2} showing how to add a Cohen function along an $\omega^2$-sequence. Later, in \S\ref{SubsectionAddingCohenOmega1} we handle the case of interest; namely, we show how to add a Cohen function along an $\omega_1$-sequence.

\subsection{Adding a Cohen function along an $\omega^2$-sequence}\label{SubsectionAddingOverOmega2}

Let us denote our ground model by $V_0$. For the rest of this section, we shall suppose that the GCH holds in $V_0$ and that this latter model accommodates a measurable cardinal $\kappa$  with $o(\kappa)=2$. Fix $U_0\lhd U_1$ normal measures over $\kappa$.  We begin performing the preparation from \cite{TomMaster} -- similarly to what we already did in Lemma~\ref{Preparation}. Namely, 
we force with the Easton-supported iteration $\mathbb{P}_\kappa$ forcing with the Lottery sum of $\Add(\alpha,1)$ and $\{\one\}$ for inaccessibles $\alpha<\kappa$. 

Suppose that $G\subseteq \mathbb{P}_\kappa$ is $V_0$-generic. Then we can lift $j_{U_1}\colon V_0\rightarrow M_{U_1}$ to $$j^*_{U_1}\colon V_0[G]\rightarrow M_{U_1}[j^*_{U_1}(G)]\s V_0[G]$$ by letting the lottery to force trivially at $\kappa$. Standard arguments show that this is the ultrapower embedding by a normal measure $W_1$ extending $U_1$. 

Let us write $j^*_{U_1}(G)=G\ast G_{(\kappa,j_{U_1}(\kappa))}.$ Arguing as in \cite{TomMaster}  we lift the measure $U_0$ (within $M_{U_1}[G]$) to a non-normal measure $W_0$ such that: 

\begin{setup}\label{Setup}\hfill
\begin{enumerate}
    \item $\mathrm{Cub}^{V_0[G]}_\kappa\s W_0$;
    \item Forcing with the Tree Prikry forcing $\mathbb{T}_{W_0}$ yields a map $$f^*_\kappa\colon \kappa\rightarrow \kappa$$ such that if $\langle \kappa_n\mid n<\omega\rangle$ is a Tree-Prikry generic sequence then $$\textstyle f^*_\kappa:=\bigcup_{n<\omega} f_{\kappa_n}\restriction [\kappa_{n-1},\kappa_n)$$ is $M_{U_1}[G]$-generic for $\Add(\kappa,1)^{V[G]}$.
\end{enumerate}
\end{setup}

  Notice that $j^*_{U_1}(\mathbb{P}_\kappa)/G$ does not add subsets to $\kappa$ and as a result $W_0$ remains a measure on $M_{U_1}[j^*_{U_1}(G)]$ which contains the club filter $\mathrm{Cub}_\kappa^{V_0[G]}$. 
  
  Note that $f^*_\kappa$ remains generic over $M_{U_1}[j^*_{U_1}(G)]$ because the tail forcing does not add new dense open subsets to the forcing. Similarly, the same applies to $V_0[G]$ as $M_{U_1}[j^*_{U_1}(G)]$ and $V_0[G]$ agree on  $(V_0[G])_{\kappa+1}$.

\smallskip

In summary, we have produced two measures $W_0\lhd W_1$ such that $W_1$ is normal, $W_0$ is non-normal yet contains $\mathrm{Cub}^{V_0[G]}_\kappa$ (i.e., $W_0$ is a $Q$-point) and forcing with $\mathbb{T}_{W_0}$ over $V_0[G]$ introduces an $\Add(\kappa,1)^{V_0[G]}$-generic (i.e., $f^*_\kappa$). 

\begin{conv}
  Herefarter we denote by $V$ the prepared model $V_0[G]$.
\end{conv}

We follow Gitik's work \cite[\S3]{GitikNonStationary} closely.  
We need a further preparation over $V$. Let $\alpha\mapsto W_{0,\alpha}$ be a function representing $W_0$ in $M_{W_1}$; namely, $j_{U_1}(\alpha\mapsto W_{0,\alpha})(\kappa)=W_0$. Let $A\in W_1$ witnessing the following: 
\begin{enumerate}\label{BlanketAssumptions}
	\item $W_{0,\alpha}$ is a measure on $\alpha$ and if $b_\alpha:=\langle \kappa^\alpha_n\mid n<\omega\rangle$ is generic for $\mathbb{T}_{W_{0,\alpha}}$  $$\textstyle \one\forces_{\mathbb{T}_{W_{0,\alpha}}}\text{``$\dot{f}^*_\alpha=\bigcup_{n<\omega}f_{\dot{\kappa}^\alpha_n}\restriction[\dot{\kappa}^\alpha_{n-1},\dot{\kappa}^\alpha_n)$ is $V$-generic for $\Add(\alpha,1)^V$''.}$$
	\item $\alpha\notin j_{W_{0,\alpha}}(A\cap \alpha)$.
\end{enumerate} 
\begin{remark}
 To get this set $A\in W_1$ it suffices to  taking any $A\in U_1\setminus U_0$ and intersect it with the collection of all $\alpha<\kappa$ for which (1) holds.   
\end{remark}

Let $\mathbb{G}_\kappa$ be the Easton-supported iteration defined recursively as follows. The iteration just forces non-trivially at measurables $\alpha\in A$. Suppose that $\mathbb{G}_\alpha$ has been defined. If $\alpha$ is a successor point of $A$ then $|\mathbb{G}_\alpha|<\alpha$ and $W_{0,\alpha}$ lifts naturally to a $V^{\mathbb{G}_\alpha}$-measure $\overline{W}_{0,\alpha}$. In that case the $\alpha$th-stage of the iteration is declared to be $\mathbb{T}_{\overline{W}_{0,\alpha}}$. Alternatively, suppose that $\alpha$ is a limit point of $A$. Once again one can lift $W_{0,\alpha}$ to a $V^{\mathbb{G}_\alpha}$-measure $\overline{W}_{0,\alpha}$ as follows:
$$(\dot{X}_
\beta)_{G_\alpha}\in \overline{W}_{0,\alpha} :\Longleftrightarrow \exists p\in G_\alpha\ (\, p^\smallfrown p_\beta\forces^{M_{W_{\alpha,0}}}_{j_{W_{\alpha,0}}(\mathbb{G}_\alpha)}[\mathrm{id}]_{W_{0,\alpha}}\in j_{W_{\alpha,0}}(\dot{X}_\beta)),$$
where $\langle p_\beta\mid \beta<\alpha^+\rangle$ is a $\leq^{*}$-increasing sequence\ in $j_{W_{\alpha,0}}(\mathbb{G}_\alpha)/G_\alpha$ with:\footnote{The key point to obtain such an $\overline{W}_{0,\alpha}$ is Clause~(2) above. Indeed, thanks to this one has that $j_{W_{0,\alpha}}(\mathbb{G}_\alpha)$ factors as a two-step iteration $\mathbb{G}_\alpha\ast \mathbb{G}_{tail},$ where the latter is an $\alpha^+$-closed iteration with respect to the corresponding Prikry order $\leq^*$.} 
\begin{enumerate}
    \item [$(i)$] $p_\beta$ decides the sentence $``[\mathrm{id}]_{W_{\alpha,0}}\in j_{W_{\alpha,0}}(\dot{X}_\beta)$'' where $\langle \dot{X}_\beta\mid \beta<\alpha^+\rangle$ is an enumeration of all $\mathbb{G}_\alpha$-names for subsets of $\alpha$.
    \item [$(ii)$] $\langle p_\beta\mid \beta<\alpha^+\rangle$ is chosen to be minimal with respect to some well-ordering of a big enough fragment of $V$ (see \cite[\S2]{GitikNonStationary} for details).
\end{enumerate}
Finally we declare the $\alpha$th-stage of the iteration to be $\mathbb{T}_{\overline{W}_{0,\alpha}}$.   

\smallskip

The above yields the preparatory Gitik's iteration $\mathbb{G}_\kappa$. Let $G\s \mathbb{G}_\kappa$ be $V$-generic and let us extend the $V$-measures $W_0$ and $W_1$ to measures $\overline{W}_0$
 and $\overline{W}_1$ in $V[G]$. Once these measures $\overline{W}_0$
 and $\overline{W}_1$ are obtained we shall define (in $V[G]$) a poset $\mathbb{P}(\kappa,2)$ such that forcing over $V[G]$ produces:
 \begin{enumerate}
     \item An $\omega^2$-sequence $\langle \kappa_\alpha\mid \alpha<\omega^2\rangle$ converging to $\kappa$;
     \item A $V[G]$-generic function for $\Add(\kappa,1)^{V[G]}$.
 \end{enumerate}

First, since $A\notin W_0$ we have that $\kappa\notin j_{W_0}(A)$ so, as before,  $W_0$ extends to $\overline{W}_0$. Second let us show how to lift $W_1$ to $\overline{W}_1$. For this let us fix $\pi\colon \kappa\rightarrow\kappa$  such that $j_{\overline{W}_0}(\pi)([\mathrm{id}]_{\overline{W}_0})=\kappa.$ 
\begin{definition}
    A sequence of ordinals $\langle \alpha_0,\dots, \alpha_n\rangle\in [\kappa]^{<\omega}$ is called \emph{$\pi$-increasing} if $\alpha_i<\pi(\alpha_{i+1})$ for all $i<n.$
\end{definition}

\begin{definition}
	For each $\pi$-increasing sequence $t\in[\kappa]^{<\omega}$  define $$\overline{W}_1(t):=\{(\dot{X}_\alpha)_G\mid \exists p\in G\,\exists\dot{T}\; (p^\smallfrown \{\langle t, \dot{T}\rangle\}^\smallfrown p_\alpha\forces^{M_{W_1}}_{j_{W_1}(\mathbb{G}_\kappa)}\kappa\in j_{W_1}(\dot{X}_\alpha))\},$$ 
	where $\langle \dot{X}_\alpha\mid \alpha<\kappa^+\rangle$ and $\langle p_\alpha\mid\alpha<\kappa^+\rangle$ are as in \cite[\S3]{GitikNonStationary}.
	\end{definition}
\begin{remark}
    By our inductive construction the $\kappa$th-stage of the iteration $j_{W_1}(\mathbb{G}_\kappa)$ is  exactly the Tree Prikry forcing $\mathbb{T}_{\overline{W}_0}.$ For this one has to argue that the lifting of the measure $W_0$ is the same  both when computed in $V[G]$ and in $M_{W_1}[G]$. This is where the well-ordering of the universe plays an essential role. We defer to provide further details about this aspect and instead refer our readers to \cite[Lemma~2.1]{GitikNonStationary}.
\end{remark}
 
	It is not hard to check that $\overline{W}_1(t)$ is a measure in $V[G]$ concentrating on
	$$\{\alpha<\kappa\mid \text{``The Prikry sequence $b_\alpha$ for $\mathbb{T}_{\overline{W}_{0,\alpha}}$ over $V[G_\alpha]$ end-extends $t$''}\}.$$
	In addition the following properties hold upon  $\overline{W}_1(t)$: 
 \begin{enumerate}
     \item $\overline{W}_{1}(t)$ is not normal as it concentrates on singular cardinals.
     \item  Since $\mathbb{G}_\kappa$ is $\kappa$-cc, $W_1$ is normal and $W_1\s \overline{W}_{1}(t)$,
      $$(\mathrm{Cub}_\kappa)^{V[G]}\s \overline{W}_{1}(t).$$
 \end{enumerate}

 Using $\overline{W}_0$ and $\langle \overline{W}_{1}(t)\mid t\in [\kappa]^{<\omega}\,\wedge\, \text{$t$ is $\pi$-increasing}\rangle$ we present Gitik's forcing $\mathbb{P}(\kappa,2)$ adding an $\omega^2$-sequence to $\kappa$ without adding bounded sets. 

 \begin{definition}
    A sequence $t=\langle \xi_0,\dots,\xi_k\rangle\in [\alpha]^{<\omega}$ is \emph{$2$-coherent} if
    \begin{enumerate}
        \item $t$ is increasing;
        \item $o^{\vec{U}}(\xi_i)\leq 1$ for all $i<k;$
        \item for all $i<k$ let $i^*\leq i$ be the first index such that $$\text{$o^{\vec{U}}(\xi_j)<o^{\vec{U}}(\xi_i)$ for all $i^*\leq j<i$.}$$ Then, $b_{\xi_i}$ end-extends $\bigcup_{i^*\leq j<i}(b_{\xi_j}\cup\{\xi_j\})$  where each $b_{\xi_\ell}$ denotes the generic sequence added by $\mathbb{T}_{\overline{W}_{\xi_\ell}}$ over $V^{\mathbb{G}_{\xi_\ell}}$. 
    \end{enumerate}
    \end{definition}
    Given a $2$-coherent sequence $t$ we denote
    $$\textstyle b_t:=\bigcup_{\xi\in r} b_\xi.$$

   Also we denote by $t\restriction 1$ the following sequence: If $o^{\vec{U}}(\max(t))= 1$ then $t\restriction \bar\beta:=\varnothing$. Otherwise, let $i^*<|t|$ be the first index with $o^{\vec{U}}(\xi_j)=0$ for all $i^*\leq j<i$ and set $t\restriction \bar{\beta}:=\langle\xi_{i^*}, \dots, \xi_{|t|-1}\rangle.$

    \begin{definition}
        A condition in $\mathbb{P}(\alpha,2)$ is a pair $\langle t, T\rangle$ where:
    \begin{enumerate}
    	\item  $t$ is $2$-coherent;
    	\item  $T$ is a tree on $[\kappa]^{<\omega}$ with trunk $\varnothing \in T$;
    	\item  $t{}^\smallfrown s$ is $2$-coherent for all $s\in T$, $\Succ_{T}(s)=\bigcup_{\bar{\beta}<2} \Succ_{T,\bar{\beta}}(s)$ and
    	$$\mathrm{Succ}_{T,0}(s)\in \overline{W}_0\;\wedge\;\mathrm{Succ}_{T,1}(s)\in \overline{W}_1((t{}^\smallfrown s)\restriction 1).$$ 
    \end{enumerate}

    Given $\langle t, T\rangle, \langle s, S\rangle\in \mathbb{P}(\alpha,2)$  write $\langle s,S\rangle \leq^* \langle t, T\rangle$ iff $t=s$ and $T\s S$.
    
    Also, say that $\langle t, T\rangle$ and $\langle s,S\rangle$ are \emph{equivalent} if $b_t=b_s$ and $T=S$. 
    \end{definition}
 
 
 Let $H\s \mathbb{P}(\kappa,2)$	a generic filter over $V[G]$. Let $C_H$ be the $\omega^2$-sequence added by $H$ and $\langle \kappa_n\mid n<\omega\rangle$  be the increasing enumeration of the limit points of $C_H$ (see Definition~\ref{GenericObject}). Then $$\textstyle C_H=\bigcup_{n<\omega}b_{\kappa_n}\cup\{\kappa_n\}.$$ 
 For each $n<\omega$ 
 the Tree Prikry generic $b_{\kappa_n}$ for $\mathbb{T}_{\overline{W}_{0,\kappa_n}}$ (over $V[G_{\kappa_n}]$) is, by the Mathias criterion for the Tree Prikry forcing \cite{TomTreePrikry}, $V$-generic for $\mathbb{T}_{W_{0,\kappa_n}}$. Thus, by our Clause~(1) in page~\pageref{BlanketAssumptions}, this  generates a  $V$-generic Cohen function $f^*_{\kappa_n}$ for $\Add(\kappa_n,1)^V.$ 
 \begin{lemma}
     $f^*_{\kappa_n}$ induces a $V[G_{\kappa_n}]$-generic
 Cohen  for $\Add(\kappa_n,1)^{V[G_{\kappa_n}]}$. 
 \end{lemma}
 \begin{proof}
  Since $\mathbb{G}_{\kappa_n}$ is an $\kappa_n$-cc forcing of size $\kappa_n$, the poset $\Add(\kappa_n,1)^V$ is isomorphic to the term-space forcing $\mathbb{A}(\mathbb{G}_\alpha, \dot{\Add}(\alpha,1))$ (see \cite[p.9]{CumGCH}). Thus, modulo isomorphisms, $f^*_{\kappa_n}$ is $V$-generic for this latter poset. By standard arguments about the term space forcing (see e.g. \cite[Proposition~22.3]{Cummings-handbook}), $f^*_{\kappa_n}$ and $G_{\kappa_n}$ together induce a $V[G_{\kappa_n}]$-generic filter for $\Add(\kappa_n,1)^{V[G_{\kappa_n}]}$. This completes the verification of the lemma.
 \end{proof}
For simplicity, let us keep calling $f^*_{\kappa_n}$ the generic for $\Add(\kappa_n,1)^{V[G_{\kappa_n}]}$.
\begin{lemma}\label{Genericityincase2}
    $f^*_\kappa:=\bigcup_{n<\omega} f^*_{\kappa_n}\restriction [\kappa_{n-1},\kappa_n)$ is $V[G]$-generic for $\Add(\kappa,1)^{V[G]}$.
\end{lemma}
\begin{proof}
    Let $\mathcal{A}\in V[G]$ be a maximal antichain for $\Add(\kappa,1)^{V[G]}$. Consider the function $f\colon\kappa\rightarrow \kappa$ defined in $V[G]$ as follows. For each $p\in \Add(\alpha,1)^{V[G]}$ let $\beta(p)<\kappa$ be the least ordinal for which there is   $q_p\in \mathcal{A}\cap \Add(\beta(p),1)^{V[G]}$ compatible with $p$. Set $f(\alpha):=\sup_{p\in \Add(\alpha,1)^{V[G]}}\beta(p).$

    Let $C$ be the club of closure points of $f$. 
    Since $C, A\in \bigcap_{t\in[\kappa]^{<\omega}}\overline{W}_1(t)$ it follows that $\langle\kappa_n\mid n\geq n_0\rangle\s A\cap C$ for some $n_0<\omega$. Let $\kappa_n$ be one of such ordinals. Note that $\mathcal{A}\cap \Add(\kappa_n,1)^{V[G]}=\mathcal{A}\cap \Add(\kappa_n,1)^{V[G_{\kappa_n}]}$ is a maximal antichain for $\Add(\kappa_n,1)^{V[G_{\kappa_n}]}$. By further shrinking $A\cap C$ we may assume (as the next claim demonstrates) that $\mathcal{A}\cap \Add(\kappa_n,1)^{V[G_{\kappa_n}]}\in V[G_{\kappa_n}]$:
    \begin{claim}
$\{\alpha<\kappa\mid \mathcal{A}\cap \Add(\alpha,1)^{V[G_\alpha]}\in V[G_\alpha]\}\in \overline{W}_1(t)$ for all  $t$.
    \end{claim}
    \begin{proof}[Proof of claim]
    Fix an arbitrary $\pi$-increasing sequence $t$. Fix $\dot{\mathcal{A}}$ and $\dot{X}$,  $\mathbb{G}_\kappa$-names for $\mathcal{A}$ and the above-displayed set, respectively. Let $p\in G$ forcing the above properties about $\dot{\mathcal{A}}$ and $\dot{X}.$
    We can moreover assume  that $\dot{\mathcal{A}}\s V_\kappa$ because $\mathbb{G}_\kappa$ is $\kappa$-cc and $\mathbb{G}_\kappa\s V_\kappa.$ In particular, $j_{W_1}(\dot{\mathcal{A}})\cap V_\kappa=\dot{\mathcal{A}}\in M_{W_1}$ and thus $\dot{\mathcal{A}}_G\in M_{W_1}[G]$. Note that this is still true in any generic extension of $M_{W_1}[G]$ by the tail forcing $j_{W_1}(\mathbb{G}_\kappa)/G$. Therefore, there are  $p\leq q\in G$,  $\langle t, \dot{T}\rangle\in \mathbb{P}(\kappa,2)$ and $p_\alpha$ with $$q\cup\{\langle t, \dot{T}\rangle\}\cup p_\alpha\forces_{j_{W_1}(\mathbb{G}_\kappa)}j_{W_1}(\dot{\mathcal{A}})\cap \Add(\kappa,1)^{V[\dot{G}]}\in M_{W_1}[\dot{G}].$$
    Since $j_{W_1}(q)=q$ forces the same relationship between $j_{W_1}(\dot{\mathcal{A}})$ and $j_{W_1}(\dot{X})$, the above shows that $q\cup\{\langle t, \dot{T}\rangle\}\cup p_\alpha$ forces $``\kappa\in j_{W_1}(\dot{X})$''. By definition, this is the same as saying that $\dot{X}_G\in \overline{W}_1(t)$.
    \end{proof}
    So, $\mathcal{A}$ must include a restriction of the function $f^*\restriction\kappa_n$ in that this is a bounded modification of $f^*_{\kappa_n}$, which was generic over $V[G_{\kappa_n}]$. Thus, the antichain  $\mathcal{A}$ intersects $f^*$ and we are done.
\end{proof}
	

 \subsection{Adding a Cohen function along an $\omega_1$-sequence}\label{SubsectionAddingCohenOmega1}
As in the previous section our ground model will be denoted by $V_0$ and we shall assume that both the GCH holds and that the model accomodates a Mitchell-increasing sequence $\l U_i\mid i<\omega_1\r$  of normal measures over $\kappa$. 
Again, we perform the same forcing preparation $\mathbb{P}_\kappa$ of \S\ref{SubsectionAddingOverOmega2} based on the lottery sum of the trivial forcing and $\Add(\alpha,1)$ for all inaccessibles $\alpha<\kappa$. 

\smallskip

 Let $G\s \mathbb{P}_\kappa$ be $V_0$-generic. 
\begin{lemma}
    In $V_0[G]$, $U_i$ extends to a $\kappa$-complete ultrafilter $W_i$ such that:
    \begin{enumerate}
        \item $\langle W_i\mid i<\omega_1\rangle$ is Mitchell increasing;
        \item $W_i$ is normal except whenever $i=0$;
        \item $(\mathrm{Cub}_\kappa)^{V_0[G]}\s W_0$ and $W_0$ is such that forcing with the Tree-Prikry forcing $\mathbb{T}_{W_0}$ over $V_0[G]$ introduces an $\Add(\kappa,1)^{V_0[G]}$-generic.
    \end{enumerate}
\end{lemma}
\begin{proof}
     We define a sequence a generics $\l G_i\mid i<\omega_1\r$ so that $G_i\restriction\kappa=G$ and
  $$ \text{
     $G_i\in M_{U_{i+1}}[G]$ is $(M_{U_i})^{M_{U_{i+1}}}$-generic for $j_{U_i}(\mathbb{P}_\kappa)$.}$$
        The point is the following: from the perspective of $M_{U_{i+1}}[G]$, $j_{U_i}(\mathbb{P}_\kappa)/G$ is a forcing of cardinality $\kappa^+$ and there are only $\kappa^+$-many maximal antichains to meet. In addition, by the usual arguments involving the commutative diagram between $j_{U_i}$ and $j_{U_{i+1}}$, both $M_{U_i}$ and $(M_{U_{i}})^{M_{U_{i+1}}}$ agree on  $(V_0)_{j_{U_i}(\kappa)+1}$ and therefore $G_{i}$ is $M_{U_i}$-generic. Note that $G_i\in M_{U_j}[G]$ for all $i<j$.  
        \smallskip

 For each $0<i<\omega_1$  lift $j_{U_i}\subseteq j^*_i:V_0[G]\rightarrow M_{U_i}[G_i]$ and let $W_i\in V_0[G]$ be the lifted measure. Clearly, $M_{W_i}=M_{U_i}[G_i]$. For $i=0$, we lift the second iteration $j_{U_0^2}\subseteq j_{W_0}:V[G]\rightarrow M_{U^2}[G_0]$ so that $W_0$ concentrate on Cohens. Namely, in the case where $i=0$ the measure $W_0$  will be non-normal, yet it will satisfy the blanket assumptions described in Setup~\ref{Setup} of page~\pageref{Setup}.

\begin{claim}
    $W_{i}\in M_{W_{\ell}}$ for all $\ell>i$.  
\end{claim}
\begin{proof}[Proof of claim]
Let us first assume that $\ell>i$, and consider the standard commutative diagram between the measures $U_{\ell}$ and $U_i$; namely, 
\begin{displaymath}
\begin{tikzcd}
 & V_0 \arrow{r}{j_{U_\ell}} \arrow{d}{j_{U_i}} & M_{U_\ell} \arrow{d}{(j_{U_i})^{M_{U_\ell}} } & \\
& M_{U_i} \arrow{r}{j_{U_i}(U_\ell)} & N,  
\end{tikzcd}
\end{displaymath}
where $(j_{U_i})^{M_\ell}$ stands for the ultrapower embedding by $U_i$ over $M_{U_\ell}$.

Since $W_i$ is the lifting of $j_{U_i}$ by the poset $j_{U_i}(\mathbb{P}_\kappa)$, $X\in W_i$ if and only if there is a $\mathbb{P}_\kappa$-name  $\dot{X}$ for a subset of $\kappa$ such that $\dot{X}_{G_i}=X$ and $$p\forces^{M_{U_i}}[\mathrm{id}]_{W_{i}}\in j_{U_i}(\dot{X})$$ for some $p\in G_i$. Thus, note that $W_i$ is definable via $G_i$ and $j_{U_i}$. 

On the one hand, $G_i\in M_{U_\ell}[G]\s M_{W_\ell}$. On the other hand,  by $\kappa$-ccness of $\mathbb{P}_\kappa$ every $\mathbb{P}_\kappa$-name  $\dot{X}$ for a subset of $\kappa$ can be assumed to be a member of $(V_0)_{\kappa+1}$. We shall next show that $$(j_{U_i})^{M_{U_\ell}}\restriction (V_0)_{\kappa+1}=j_{U_i}\restriction (V_0)_{\kappa+1}$$ which combined with our previous comments will establish $W_i\in M_{W_\ell}$.

To simplify notations, let us denote $j_{\ell,i}:=(j_{U_i})^{M_{U_\ell}}$ and $j_{i,\ell}:=j_{U_i}(U_\ell)$. Fix $P\s (V_{0})_{\kappa}$. Since $\crit(j_{i,\ell})=j_{U_i}(\kappa)$ we have 
$$j_{U_i}(P)=j_{i,\ell}(j_{U_i}(P))\cap (M_{U_i})_{j_{U_i}(\kappa)}=j_{i,\ell}(j_{U_i}(P))\cap N_{j_{U_i}(\kappa)}.$$
By commutativity of the diagram this amounts to saying
$$j_{U_i}(P)=j_{\ell,i}(j_{U_\ell}(P))\cap N_{j_{U_i}(\kappa)}=j_{\ell,i}(j_{U_\ell}(P)\cap (M_{U_\ell})_\kappa).$$
Once again, since $\crit(j_{U_\ell})=\kappa$, $(M_{U_\ell})_\kappa=(V_0)_\kappa$ and  
$$j_{U_\ell}(P)\cap (M_{U_\ell})_\kappa=j_{U_\ell}(P)\cap (V_0)_\kappa=P.$$
Combining these two latter equations we obtain
$$j_{U_i}(P)=j_{\ell,i}(P),$$
as needed.

The proof for $i=0$ is identical, bearing in mind that if $U_0\in M_{U_{\ell}}$ then also $U_0^2\in M_{U_{\ell}}$ and thus the argument for $W_0$ and $W_{\ell}$ is the same as the one of the previous paragraph, working with the commutative diagram of $U_0^2$  and $U_{\ell}.$ 
\end{proof}
The above completes the proof of the lemma.
\end{proof}

\begin{setup}\label{setup2}
   We denote our new ground model by $V$. Invoking Corollary \ref{Cor:almostCohereSequence} inside $V$ we derive an almost coherent sequence $\vec{U}$ on $\kappa$ of length $\omega_1$ such that $U(\kappa,i)=W_i$ (see Definition~\ref{Def:almostcoherentesequence}). By Changing $\vec{U}$ on a null-set, we may assume that for every measurable cardinal $\alpha<\kappa$,  $U(\alpha,0)$ is a non-normal $\alpha$-complete ultrafilter witnessing the clauses provided in Setup~\ref{Setup}.
\end{setup}

\begin{definition}
For each $i<\omega_1$ 
define $$\textstyle \dom_1(\vec{U}):=\{\eta\leq\kappa\mid o^{\vec{U}}(\eta)>0\}.$$ 
\end{definition}

As in the previous section we begin defining an Easton-supported iteration $$\mathbb{G}_\kappa:=\varinjlim \langle \mathbb{G}_\alpha;\dot{\mathbb{Q}}_\alpha\mid \alpha<\kappa\rangle$$ using the almost coherent sequence $\vec{U}:=\langle \vec{U}(\alpha,\beta)\mid \alpha\leq \kappa,\, \beta<o^{\vec{U}}(\alpha)\rangle$. 

\smallskip
Fix $\alpha\leq \kappa$ and suppose that $\mathbb{G}_\alpha$ has been defined. If $\alpha\notin\dom_1(\vec{U})$  we declare $\mathbb{Q}_\alpha$ to be the trivial forcing. Otherwise, $\alpha\in\dom_1(\vec{U})$ and we have two options: either $\dom_1(\vec{U})\cap \alpha$ is bounded in $\alpha$ or it is not. In the former case, $o^{\vec{U}}(\alpha)=1$ and by standard arguments due to L\'evy and Solovay, $U(\alpha,0)$ extends to a measure $\bar{U}(\alpha,0)$ in $V^{\mathbb{P}_\alpha}$. In this latter case we declare $\mathbb{Q}_\alpha$ to be $\mathbb{T}_{\bar{U}(\alpha,0)}$, the Tree Prikry forcing relative to $\bar{U}(\alpha,0)$.

\smallskip

So, suppose that $\dom_1(\vec{U})\cap \alpha$ is unbounded in $\alpha$. We define  sequences $$\text{$\langle \mathbb{P}(\alpha,\beta)\mid \beta\leq o^{\vec{U}}(\alpha)\rangle$ and $\langle U(\alpha, \beta, t)\mid \beta<o^{\vec{U}}(\alpha),\, t\in\mathcal{C}_{\alpha,\beta}\rangle$}$$ as follows. Let $\mathbb{P}(\alpha,0)$ be the trivial forcing and $U(\alpha,0, \varnothing)$ the measure defined as follows. Let $j^\alpha_0\colon V\rightarrow N^\alpha_0$ be the ultrapower embedding by $U(\alpha,0)$. By coherency, $j^\alpha_0(\vec{U})\restriction\alpha+1=\vec{U}\restriction (\alpha,0)=\varnothing$. Hence, $j^\alpha_0(\mathbb{G}_\alpha)$ factors as $$\mathbb{G}_\alpha\ast \{\varnothing\}\ast \mathbb{G}_{(\alpha,j^\alpha_0(\alpha))}.$$ The tail forcing $\mathbb{G}_{(\alpha,j^\alpha_0(\alpha))}$ has the Prikry property and is $\alpha^{++}$-closed with respect to the $\leq^*$-order. Standard arguments  allow us to  produce an extension $U(\alpha,0,\varnothing)$ of $U(\alpha,0)$ in $V^{\mathbb{G}_\alpha}$. Note that $U(\alpha,0,\varnothing)$ extends the club filter $\mathrm{Cub}_\alpha$ as computed in $V^{\mathbb{G}_\alpha}$: Indeed, $U(\alpha,0)$ extends the $V$-club filter and $\mathbb{P}_\alpha$ is $\alpha$-cc (so every $V^{\mathbb{G}_\alpha}$-club contains a $V$-club). To make the forthcoming construction work smoothly we follow Gitik's ideas \cite[\S3]{GitikNonStationary} 
and define $U(\alpha,0,\varnothing)$ relative to a fix well-ordering of a large-enough fragment of the set-theoretic universe. More precisely, we define $U(\alpha,0,\varnothing)$ analogously to $\overline{W}_{\alpha,0}$ in page~\pageref{BlanketAssumptions}.

\smallskip

Suppose that both $\mathbb{P}(\alpha,\bar{\beta})$ and $U(\alpha,\bar{\beta}, t)$ have been constructed for all $\bar{\beta}<\beta\leq o^{\vec{U}}(\alpha)$ in $V^{\mathbb{G_\alpha}}$. To proceed we need the notion of $\beta$-coherency:
\begin{definition}
    A sequence $t=\langle \xi_0,\dots,\xi_k\rangle\in [\alpha]^{<\omega}$ is $\beta$-coherent if
    \begin{enumerate}
        \item $t$ is increasing;
        \item $o^{\vec{U}}(\xi_i)<\beta$ for all $i<k;$
        \item for all $i<k$ let $i^*\leq i$ be the first index such that $o^{\vec{U}}(\xi_j)<o^{\vec{U}}(\xi_i)$ for all $i^*\leq j<i$. Then, $b_{\xi_i}$ end-extends $\bigcup_{i^*\leq j<i}(b_{\xi_j}\cup\{\xi_j\}).$ Where $b_{\xi_j}$ is the generic sequence added by $\mathbb{P}(\xi_j,o(\xi_j))$ over $V^{\mathbb{G}_{\xi_j}}$. 
    \end{enumerate}
    Denote by $\mathcal{C}_{\alpha,\beta}$ the collection of all $\beta$-coherent sequences in $[\alpha]^{<\omega}$. Given $t, s\in \mathcal{C}_{\alpha,\beta}$ we say that $t$ and $s$ are \emph{equivalent} if $b_t=b_s$ where
    $$\textstyle b_r:=\bigcup_{\xi\in r} b_\xi\;\text{for $r\in \{t,s\}$}.$$

    For each $\bar{\beta}<\beta$ denote by $t\restriction \bar{\beta}$ the following sequence: If $o^{\vec{U}}(\max(t))\geq \bar{\beta}$ then $t\restriction \bar\beta:=\varnothing$. Otherwise, let $i^*<|t|$ be the first index with $o^{\vec{U}}(\xi_j)<\bar{\beta}$ for all $i^*\leq j<i$ and set $t\restriction \bar{\beta}:=\langle\xi_{i^*}, \dots, \xi_{|t|-1}\rangle.$
\end{definition}

We can now define the poset $\mathbb{P}(\alpha,\beta)$:

\begin{definition}\label{Gitikforcing}
    A condition in $\mathbb{P}(\alpha,\beta)$ is a pair $\langle t, T\rangle$ where:
    \begin{enumerate}
    	\item  $t\in\mathcal{C}_{\alpha,\beta}$;
    	\item  $T$ is a tree on $[\alpha]^{<\omega}$ with trunk $\varnothing \in T$;
    	\item  $t{}^\smallfrown s$ is $\beta$-coherent for all $s\in T$, $\Succ_{T}(s)=\bigcup_{\bar{\beta}<\beta} \Succ_{T,\bar{\beta}}(s)$ and
    	$$\mathrm{Succ}_{T,\bar{\beta}}(s)\in U(\alpha,\bar\beta, (t{}^\smallfrown s)\restriction \bar{\beta}).$$ 
    \end{enumerate}

    Given $\langle t, T\rangle, \langle s, S\rangle\in \mathbb{P}(\alpha,\beta)$  write $\langle s,S\rangle \leq^* \langle t, T\rangle$ iff $t=s$ and $T\s S$.
    
    Also, say that $\langle t, T\rangle$ and $\langle s,S\rangle$ are \emph{equivalent} if $b_t=b_s$ and $T=S$. 
\end{definition}
\begin{remark}
	Note that, formally speaking, $\Succ_{T,\bar\beta}(\cdot)$ depends also on the entire condition $\langle t, T\rangle$. To avoid overcomplicated notations we shall keep denoting the set of successors in that way, in place of $\Succ_{\langle t, T\rangle,\bar{\beta}}(\cdot)$.
\end{remark}

\begin{definition}[Minimal extensions]
	For $\langle t, T\rangle\in \mathbb{P}(\alpha,\beta)$ and $\langle \nu\rangle\in T$,  $$\langle t, T\rangle\cat\langle \nu\rangle:=\langle t{}^\smallfrown \langle\nu\rangle, T_{\langle\nu\rangle}\setminus V_{\nu+1}\rangle.$$ As customary, $T_{\langle\nu\rangle}:=\{s\in T\mid \langle \nu\rangle{}^\smallfrown s\in T\}.$
	
	In general for $\vec\nu \in T$ define $\langle t, T\rangle\cat\vec\nu$ by recursion on the length of $\vec\nu$.
\end{definition}
The standard order of $\mathbb{P}(\alpha,\beta)$ is defined as a combination of $\leq^*$ and $\cat\vec\nu$:
\begin{definition}
	For $\langle t, T\rangle, \langle s, S\rangle\in\mathbb{P}(\alpha,\beta)$ write $\langle t, T\rangle\leq \langle s, S\rangle$ if and only if there is $\vec\nu\in S$ such that $\langle t, T\rangle$ is equivalent to a $\leq^*$-extension of $\langle s{}, S\rangle\cat\vec\nu$.
\end{definition}
\begin{remark}
	If $\langle t, T\rangle$ and $\langle s, S\rangle$ are equivalent then $\langle t, T\rangle\leq \langle s,S\rangle$ and $\langle s, S\rangle\leq \langle s,S\rangle$. Thus both conditions force the same information. 
\end{remark}
Next, we define the measures $\langle U(\alpha,\beta, t)\mid t\in \mathcal{C}_{\alpha,\beta}\rangle$ as follows:
\begin{definition}
    For each $t\in \mathcal{C}_{\alpha,\beta}$,   define $$U(\alpha,\beta,t):=\{(\dot{X}_\alpha)_{G_\alpha}\mid \exists p\in G_\alpha\,\exists\dot{T}\; (p^\smallfrown \{\langle t, \dot{T}\rangle\}^\smallfrown p_\gamma\forces_{j^\alpha_\beta(\mathbb{G}_\alpha)}\alpha\in j^\alpha_\beta(\dot{X}_\gamma))\},$$ 
	where $\langle \dot{X}_\gamma\mid \gamma<\alpha^+\rangle$ and $\langle p_\gamma\mid\gamma<\alpha^+\rangle$ are as in \cite[\S3]{GitikNonStationary} and $j^\alpha_\beta$ denotes the ultrapower embedding by $U(\alpha,\beta)$. 
\end{definition}

The above completes the inductive definition of $$\text{$\mathbb{P}(\alpha,\beta)$ and $\langle U(\alpha,\beta, t)\mid t\in \mathcal{C}_{\alpha,\beta}\rangle$}$$ for all $\alpha\leq \kappa$ and $\beta\leq o^{\vec{U}}(\alpha)$. Finally let $\dot{\mathbb{Q}}_\alpha$  a $\mathbb{G}_\alpha$-name for $\mathbb{P}(\alpha, o^{\vec{U}}(\alpha)).$

\smallskip

Gitik showed  that $\langle \mathbb{P}(\alpha, o^{\vec{U}}(\alpha)),\leq,\leq^*\rangle$ is a Prikry-type forcing \cite[Lemma~3.11]{GitikNonStationary}. It is also easy to show that $\mathbb{P}(\alpha, o^{\vec{U}}(\alpha))$ is $\alpha^+$-cc and that $\langle  \mathbb{P}(\alpha, o^{\vec{U}}(\alpha)),\leq^*\rangle$ is an $\alpha^+$-closed forcing. Thus, forcing with $ \mathbb{P}(\alpha, o^{\vec{U}}(\alpha))$ does not collapse cardinals. However, forcing with $\mathbb{P}(\alpha, o^{\vec{U}}(\alpha))$ adds a cofinal sequence to $\alpha$ with order-type  $\omega^{o^{\vec{U}}(\alpha)}$. As a result this forcing changes the cofinality of $\alpha$ -- details are provided below.

\begin{definition}\label{GenericObject}
    Let $H\subseteq \mathbb{P}(\alpha,o^{\vec{U}}(\alpha))$ a $V[G_\alpha]$-generic. Define 
    $$\textstyle b_\alpha:=\bigcup\{ b_\beta\mid \exists \l t,T\r\in H, \   \beta\in t\}.$$
\end{definition}
Note that if $\l t,T\r\leq \l s,S\r$ then $s$ is equivalent to an initial segment of $t$ and therefore $b_t$ end-extends $b_s$. It follows that for each $\l t,T\r\in H$, $b_\alpha$ end-extends $b_t$. Arguing inductively, one can now prove that $b_\alpha$ is a club with $\mathrm{otp}(b_\alpha)=\omega^{o^{\vec{U}}(\alpha)}$. It follows that the cofinality of $\alpha$ in $V[G_\alpha]$ is determined by this order-type, and in particular we have the following:
\begin{corollary}
    Let $G_\kappa$ be $V$-generic for $\mathbb{P}_\kappa$ and let $G$ be $V^*=V[G_\kappa]$-generic for $\mathbb{P}(\kappa,\omega_1)$. Then $\cf^{V^*[G]}(\kappa)=\omega_1$.
\end{corollary}


\smallskip

Let $G_\kappa$ be $V$-generic for $\mathbb{P}_\kappa$ and let $V^*=V[G_\kappa]$. By definition of the iteration $\mathbb{P}_\kappa$, for every $\alpha\in \dom_1(\vec{U})$, we have a $V[G_\alpha]$-generic sequence $b_\alpha$ for $\mathbb{P}(\alpha,o^{\vec{U}}(\alpha))$. Note that every $\xi\in b_\alpha$ with $\vec{U}$-order $0$ is in some $b_\gamma$ for $\gamma\leq\alpha$ with $o^{\vec{U}}(\gamma)=1$, and by definition, $b_\gamma$ is generic for $\mathbb{T}_{\bar{U}(\alpha,0)}$. It follows that $f_\xi$ (the $V_0$-generic functions for $\Add(\xi,1)$) is defined. In particular, we may assume that if $C_G=\l \kappa_i\mid i<\omega_1\r$ is a $V^*$-generic filter for $\mathbb{P}(\kappa,\omega_1)$, then for every $i<\omega_1$, $\kappa_{i+1}\in Y_0$. And in particular $f_{\kappa_{i+1}}:\kappa_{i+1}\rightarrow \kappa_{i+1}$.
\begin{theorem}\label{GitikProjects}
Let $G\subseteq \mathbb{P}(\kappa,\omega_1)$ be $V^*$-generic and let $C_G=\l \kappa_i\mid i<\omega_1\r$ be the generic club sequence. Then, 
$$\textstyle f^*:=f_0\restriction\kappa_0\cup\bigcup_{i<\omega_1}f_{\kappa_{i+1}}\restriction [\kappa_i,\kappa_{i+1})$$ is $V^*$-generic for $\Add(\kappa,1)^{V^*}$. 
\end{theorem}
\begin{proof}
Let us denote by $\mathrm{Succ}(C_G)$ the increasing sequence of successor points of $C_G$; namely $\langle \kappa_{i+1}\mid i<\omega_1\rangle$. For each $\alpha\in C_{G}\cup \{\kappa\}$ define
$$\textstyle f^*_\alpha:=f_0\restriction \kappa_0\cup \bigcup_{\beta\in\Succ(C_G)\cap\alpha}f_\beta\restriction[\beta^-,\beta)$$ where $\beta^{-}$ stands for the predecessor of $\beta$ in $C_G$. 

We will show that for every $\alpha\in C_G\cup\{\kappa\}$, $f^*_\alpha$ is $V[G_\alpha]$-generic for $\Add(\alpha,1)^{V[G_\alpha]}$. In particular,  $f^*=f^*_\kappa$ will be generic over $V^*=V[G_\kappa]$. 

The proof is by induction on $0<\gamma\leq\omega_1$, and the induction step is proved for all $\alpha\in C_G\cup\{\kappa\}$ with  $o^{\vec{U}}(\alpha)=\gamma$. 
For $o^{\vec{U}}(\alpha)=1$, $\mathbb{P}(\alpha,1)$ is just the Tree-Prikry forcing with $U(\alpha,0,\varnothing)\supseteq U(\alpha,0)$. In this case note that $$f^*_\alpha=f^*_{\alpha^*}\cup\bigcup_{ i<\omega}f_{\alpha_{i}}\restriction [\alpha_{i-1},\alpha_{i}),$$ where $\l \alpha_i\mid i<\omega\r$ is the Prikry sequence $b_\alpha$ added by $\mathbb{P}(\alpha,1)$, and $\alpha^*\in C_G\cap \alpha$ is the last  ordinal such that $o^{\vec{U}}(\alpha^*)\geq 0^{\vec{U}}(\alpha)$ and $o^{\vec{U}}(\beta)<o^{\vec{U}}(\alpha)$ for all $\beta\in C_G\cap (\alpha^*,\alpha)$. The sequence $\langle \alpha_n\mid n<\omega\rangle$ is also $V_0$-generic for $U(\alpha,0)$ (by the Mathias criterion) {and by the construction of $U(\alpha,0)$,}  $$\bigcup_{ n<\omega} f_{\alpha_n}\restriction [\alpha_{n-1},\alpha_n)$$  is $\Add(\alpha,1)^{V_0[G\restriction\alpha]}$-generic (see Setup~\ref{setup2}). Since $V$ is a generic extension of $V_0[G\restriction\alpha]$ by an $\alpha^+$-closed forcing (namely, the tail of the preliminary lottery iteration), this function is also generic for $\Add(\alpha,1)^V$. 
\begin{claim}
    $\bigcup_{ n<\omega} f_{\alpha_n}\restriction [\alpha_{n-1},\alpha_n)$
is  
 a $V[G_{\alpha}]$-generic for  $\Add(\alpha,1)^{V[G_{\alpha}]}$. 
 
 In particular,  $f^*_\alpha$ is $V[G_\alpha]$-generic for $\Add(\alpha,1)^{V[G_\alpha]}$.
\end{claim}
 \begin{proof}[Proof of Claim.] First we note that $V[G_\alpha]$ is a forcing extension of $V$  by $\mathbb{G}_{\alpha}$, which is an $\alpha$-c.c forcing of size $\alpha$. It follows that $\Add(\alpha,1)^V$ is isomorphic to the term-space forcing $\mathbb{A}(\mathbb{G}_\alpha, \dot{\Add}(\alpha,1))$ (see \cite[p.9]{CumGCH}). Thus, modulo isomorphism, $\bigcup_{ n<\omega} f_{\alpha_n}\restriction [\alpha_{n-1},\alpha_n)$ is $V$-generic for this latter poset. By standard arguments about the term space forcing, $\bigcup_{ n<\omega} f_{\alpha_n}\restriction [\alpha_{n-1},\alpha_n)$ and $G_{\alpha}$ together induce a $V[G_{\alpha}]$-generic filter for $\Add(\alpha,1)^{V[G_{\alpha}]}$.

 For the last claim, $f^*_\alpha$ is a bounded modification of $\bigcup_{ n<\omega} f_{\alpha_n}\restriction [\alpha_{n-1},\alpha_n)$ using a function in $V[G_\alpha]$ (i.e. $f^*_{\alpha^*}$) and so $f^*_\alpha$ is also $V[G_\alpha]$-generic. 
 \end{proof}

 \smallskip

Let us argue for the general case. Our induction hypothesis is
\begin{equation*}
   \text{$\forall\beta\in C_G\cap \alpha (o^{\vec{U}}(\beta)<o^{\vec{U}}(\alpha)\,\Rightarrow\, f^*_\beta$ is $\Add(\beta,1)^{V[G_\beta]}$-generic over $V[G_\beta])$. }
\end{equation*}

\begin{claim}\label{falphaisgeneric}
    $f^*_\alpha$ is $V[G_\alpha]$-generic.
\end{claim}
\begin{proof}[Proof of claim]
     Let $\mathcal{A}\in V[G_\alpha]$ be a maximal antichain for $\Add(\alpha,1)^{V[G_\alpha]}$.  Consider the function $f\colon\alpha\rightarrow \alpha$ defined in $V[G_\alpha]$ as follows. For each $\beta<\alpha$ and $p\in \Add(\beta,1)^{V[G_\alpha]}$ let $\beta(p)<\alpha$ be the least for which there is a condition  $q_p\in \mathcal{A}\cap \Add(\beta(p),1)^{V[G]}$ compatible with $p$. Set $$\textstyle f(\beta):=\sup_{p\in \Add(\beta,1)^{V[G_\alpha]}}\beta(p).$$
      Let $C$ be the club of closure points of $f$. 
     Note that for each $\beta<\alpha$ regular, $$\mathcal{A}\cap \Add(\beta,1)^{V[G_\alpha]}=\mathcal{A}\cap \Add(\beta,1)^{V[G_{\beta}]},$$ and no bounded subsets of $\alpha$ are introduced by the forcing passing from $V[G_\beta]$ to $V[G_\alpha]$.  Clearly, if $\beta\in C$ then $\mathcal{A}\cap \Add(\beta,1)^{V[G_{\beta}]}$ is a maximal antichain for $\Add(\beta,1)^{V[G_{\beta}]}$. Let us prove that for a measure-one set of $\beta$'s, $\mathcal{A}\cap \Add(\beta,1)^{V[G_{\beta}]}\in V[G_{\beta}]$. Once this is established we will  be mostly done.


    \begin{sclaim}
$X=\{\nu<\alpha\mid \mathcal{A}\cap \Add(\nu,1)^{V[G_\nu]}\in V[G_\nu]\}\in U(\alpha,\gamma,t)$ for all $\gamma<o^{\vec{U}}(\alpha)$ and all $\gamma$-coherent sequence $t\in [\alpha]^{<\omega}$.
    \end{sclaim}
    \begin{proof}[Proof of subclaim]
    Fix an arbitrary $t$. Let $\dot{\mathcal{A}}$ and $\dot{X}$ be a $\mathbb{G}_\alpha$-names for $\mathcal{A}$ and the above-displayed set, respectively. Let $p\in G_\alpha$ forcing the above about $\dot{\mathcal{A}}$ and $\dot{X}.$
    We can moreover assume  that $\dot{\mathcal{A}}\s V_\alpha$ - this is possible because $\mathbb{G}_\alpha$ is $\alpha$-cc and $\mathbb{G}_\alpha\s V_\alpha.$ In particular, $j^\alpha_\gamma(\dot{\mathcal{A}})\cap V_\alpha=\dot{\mathcal{A}}\in M_{U(\alpha,\gamma)}$. Thus, $\dot{\mathcal{A}}_{G_\alpha}\in M_{U(\alpha,\gamma)}[G_\alpha]$. This is still true in any generic extension of $M_{U(\alpha,\gamma)}[G_\alpha]$ by $j^\alpha_\gamma(\mathbb{P}_\alpha)/G_\alpha$. Therefore, there are  $q\in G_\alpha$ ($q\leq p$), $\{\langle t, \dot{T}\rangle\}\in \mathbb{P}(\alpha,\gamma)$ and $p_\nu$ such that $$q\cup\{\langle t, \dot{T}\rangle\}\cup p_\nu\forces_{j^\alpha_\gamma(\mathbb{G}_\alpha)}j^\alpha_\gamma(\dot{A})\cap \Add(\alpha,1)^{V[\dot{G}_\alpha]}\in M_{U(\alpha,\gamma)}[\dot{G}_\alpha].$$
    Since $j^\alpha_\gamma(q)=q$ forces the same connection between $j(\dot{\mathcal{A}})$ and $j(\dot{X})$, the above shows that $q\cup\{\langle t, \dot{T}\rangle\}\cup p_\nu$ forces $``\alpha\in j^\alpha_\gamma(\dot{X})$''. By definition, this is the same as saying that $\dot{X}_{G_\alpha}\in U(\alpha,\gamma,t)$.
    \end{proof}
    Since $C\cap X\in \bigcap_{t\in[\kappa]^{<\omega}}U(\alpha,\beta,t)$ it follows that there is $\alpha_0<\alpha$ such that $b_\alpha\setminus\alpha_0\s C$. For each $\beta\in b_\alpha\setminus \alpha_0$, 
    $\mathcal{A}\cap\Add(\beta,1)^{V[G_\beta]}\in V[G_\beta]$ is a maximal antichain. Hence, $\mathcal{A}\cap\Add(\beta,1)^{V[G_\beta]}$ must include a restriction of the function $f^*_\alpha\restriction\beta$, as this function is a bounded modification of $f^*_{\beta}$ which is $V[G_{\beta}]$-generic by the induction hypothesis. All in all, $\mathcal{A}$ includes a restriction of $f^*_\alpha$ and we are done.
    \end{proof}
    The proof of Claim~\ref{falphaisgeneric} completes the inductive verification and establishes the proof of Theorem~\ref{GitikProjects}.
    \end{proof}
    \begin{corollary}
        Working in $V^*$, $\mathbb{P}(\kappa,\omega_1)$ projects onto $\Add(\kappa,1)^{V^*}.$
    \end{corollary}
\section{Further Directions}\label{sec: open questions}
In this last section we should like to draw a few  future directions in which the present work could be applied. Our first proposed direction regards the existence of a minimal \emph{Sacks-like} poset that singularizes a measurable cardinal to  uncountable cofinalities. This (if feasible at all) will be analogous to the main poset devised in \cite{MinimalPrikry}. Thus, we ask:
\begin{question}
    Is there a Prikry-type forcing   that changes the cofinality of a measurable cardinal to $\omega_1$ whose generic extension does not have proper intermediate inner models?
\end{question}
It is not far-fetched that a tree-like variation of the {\color{blue} Gitik forcing} non-normal Magidor/Radin forcing presented here may work in this respect. 

\smallskip

There is another question that regards the preparation of  Lemma~\ref{Preparation} (first described in \cite{TomMaster}). This preparation forces with the lottery sum of $\{\Add(\alpha,1), \{\one\}\}$ for every inaccessible $\alpha<\kappa$ and yields  a non-normal $\kappa$-complete ultrafilter \emph{concentrating on Cohens} (Clause~(2) in Setup~\ref{Setup}). Namely, if $C=\langle \kappa_n\mid n<\omega\rangle$ is a Prikry sequence for the Tree Prikry forcing $\mathbb{T}_U$ then $f_C:=\bigcup_{n<\omega} f_{\kappa_n}\restriction [\kappa_{n-1},\kappa_n)$ is $\Add(\kappa,1)$-generic over $V$ where $f_{\kappa_n}$ are $\Add(\kappa_n,1)$-generics over an inner model of $V$ arising from the  forcing preparation. A natural  inquiry is what can be said about $f_C$ and $f_{C´}$ whenever $C$ and $C'$ are mutually generic $\mathbb{T}_U$-generic sequences. 
\begin{question}
Suppose $C_1,C_2$ partition $C$ into two infinite sets, are $f_{C_1},f_{C_2}$ mutually generic over $V$?
\end{question}


Similar techniques to the ones developed in this paper permitt to construct Mitchell increasing sequences with several non-normal ultrafilters, each of which concentrating on Cohens. 
\begin{question}
    Suppose that $U_0\triangleleft U_1$ are two non-normal ultrafilters on $\kappa$ concentrating on Cohens. What is the relation between the corresponding Cohen generic functions? 
\end{question}

\bibliographystyle{alpha}
\bibliography{citations}
\end{document}